\documentclass[journal]{IEEEtran}
\newcommand{\beq}{\begin{equation}}
\newcommand{\eeq}{\end{equation}}
\newcommand{\bsq}{\begin{subequations}}
\newcommand{\esq}{\end{subequations}}
\newcommand{\bq}{\begin{eqnarray}}
\newcommand{\eq}{\end{eqnarray}}
\newcommand{\bqn}{\begin{eqnarray*}}
\newcommand{\eqn}{\end{eqnarray*}}

\usepackage{mathptmx}       
\usepackage{helvet}         
\usepackage{courier}        
\usepackage{type1cm}        
\usepackage{makeidx}         
\usepackage{graphicx}        

\usepackage{multicol}        
\usepackage[bottom]{footmisc}
\usepackage{ifthen}
\usepackage{subfigure}
\usepackage[usenames,dvipsnames]{color}

\makeindex             

\usepackage{graphics} 
\usepackage{graphicx} 
\usepackage{epsfig} 
\usepackage{epstopdf}
\usepackage{mathptmx} 
\usepackage{times} 
\usepackage[cmex10]{amsmath}

\usepackage{mathtools}  
\usepackage{amssymb}  
\usepackage{amsthm}

\usepackage{amsfonts}
\usepackage{cite}
\usepackage{verbatim}
\usepackage{comment}
\usepackage{algorithm}
\usepackage{algorithmic}
\usepackage{multirow}
\usepackage{cite}
\usepackage{stfloats}
\usepackage{supertabular}
\usepackage{longtable}

\DeclareGraphicsExtensions{.pdf,.png,.jpg,.eps,.ps}

\usepackage{arydshln}
\usepackage{multicol}
\usepackage{colortbl}
\theoremstyle{definition}

\newtheorem{lemma}{Lemma}

\newtheorem{proposition}{Proposition}

\theoremstyle{definition}
\newtheorem{definition}{Definition}

\newtheorem{assumption}{Assumption}
\ifCLASSINFOpdf
\else
\fi

\usepackage{comment}
\usepackage{ifthen}
\newboolean{showcomments}
\setboolean{showcomments}{true}
\newcommand{\ychen}[1]{\ifthenelse{\boolean{showcomments}}
        { \textcolor{red}{(YC:  #1)}}{}}
\newcommand{\fliu}[1]{\ifthenelse{\boolean{showcomments}}
        { \textcolor{blue}{(FL:  #1)}}{}}

\usepackage{amsmath}
\allowdisplaybreaks[4]

\begin{document}

%
\title{Buy or Sell? Energy Sharing of Prosumers on Constrained Networks}
%
%
%

\author{Yue Chen,
        Shengwei Mei,~\IEEEmembership{Fellow,~IEEE,}
        Wei Wei, ~\IEEEmembership{Senior Member,~IEEE,}
        Steven H. Low,~\IEEEmembership{Fellow,~IEEE,}\\
        Adam Wierman, ~\IEEEmembership{Member,~IEEE,}
        and Feng Liu,~\IEEEmembership{Senior Member,~IEEE}
}

%
%

\markboth{Journal of \LaTeX\ Class Files,~Vol.~XX, No.~X, Feb.~2019}%
{Shell \MakeLowercase{\textit{et al.}}: Bare Demo of IEEEtran.cls for IEEE Journals}
%



\maketitle

\begin{abstract}
The advent of intelligent agents who produce and consume energy by themselves has led the smart grid into the era of ``prosumer'', offering the energy system and customers a unique opportunity to revaluate/trade their spot energy via a sharing initiative. To this end, designing an appropriate sharing mechanism is an issue with crucial importance and has captured great attention. This paper addresses the prosumers' demand response problem via energy sharing. Under a general supply-demand function bidding scheme, a sharing market clearing procedure considering network constraints is proposed, which gives rise to a generalized Nash game. The existence and uniqueness of market equilibrium are proved in non-congested cases. When congestion occurs, infinitely much equilibrium may exist because the strategy spaces of prosumers are correlated. A price-regulation procedure is introduced in the sharing mechanism, which outcomes a unique equilibrium that is fair to all participants. Properties of the improved sharing mechanism, including the individual rational behaviors of prosumers and the components of sharing price, are revealed. When the number of prosumers increases, the proposed sharing mechanism approaches social optimum. Even with fixed number of resources, introducing competition can result in a decreasing social cost. Illustrative examples validate the theoretical results and provide more insights for the energy sharing research.
\end{abstract}


%
\IEEEpeerreviewmaketitle

\section*{Nomenclature}
\addcontentsline{toc}{section}{Nomenclature}
\subsection{Indices and Sets}
\begin{IEEEdescription}[\IEEEusemathlabelsep\IEEEsetlabelwidth{ssssssss}]
\item[$i$] Index of prosumers.
\item[$k$] Index of resources.
\item[$\mathcal{I}$] Set of prosumers.
\item[$\mathcal{K}$] Set of resources.
\item[$V$] Set of buses.
\item[$E$] Set of edges.
\item[$f_i(\cdot)$] Disutility function of prosumer $i$.
\item[$s_i(\cdot)$] Sharing cost of prosumer $i$.
\item[$\Pi_i(\cdot)$] Total cost function of prosumer $i$.
\item[$X_i$] Action set of player $i$, and $X = \prod_i X_i$.
\end{IEEEdescription}
       
\subsection{Parameters}
\begin{IEEEdescription}[\IEEEusemathlabelsep\IEEEsetlabelwidth{ssssssss}]
\item[$I$] Number of prosumers.
\item[$K$] Number of resources of each prosumer.
\item[$D_i^0$] Fixed amount of energy prosumer $i$ consumes.
\item[$p_i^{k0}$] The amount of energy $i$ generates originally.
\item[$E_i^0$] The amount of energy bought from grid by $i$.
\item[$D_i$] Required load reduction of prosumer $i$.
\item[$a$] Price elasticity of prosumers
\item[$c_i^k$] Coefficients of the disutility function for prosumer $i$.
\item[$\pi_{il}$] Line flow distribution factor from bus $i$ to line $l$.
\item[$F_l$] Flowlimit for line $l$.
\end{IEEEdescription}

\subsection{Decision Variables}
\begin{IEEEdescription}[\IEEEusemathlabelsep\IEEEsetlabelwidth{ssssssss}]
\item[$p_i^k$] Output adjustment of resource $k$ for prosumer $i$.
\item[${\lambda _i}$] Sharing market clearing price for prosumer $i$.
\item[$\lambda_i^c$] Sharing price for prosumer $i$ after regulation.
\item[$b_i$] Willingness to pay/buy of prosumer $i$.
\item[$q_i$] Energy transaction in the sharing market.
\end{IEEEdescription}

\section{Introduction}
%
%
%
%

\IEEEPARstart{T}{he} advances in distributed renewable generation technology amid the cost decline of energy storage units have led to the proliferation of utility-scale wind turbines, rooftop PV panels and energy storages \cite{da2014impact}. Traditional customer is now transferring to the so-called ``prosumer''. Except for their own needs, prosumers can even supply energy to the grid if the production is abundant, and offer great potential for enhancing the efficiency of energy utilization via energy sharing \cite{rathnayaka2011innovative}. In the past decade, the sharing economy has become a heated topic in both industry and academia. Several third-party sharing platforms, such as Uber for ride-sharing \cite{cannon2014uber}, AirBnB for room-sharing \cite{zervas2017rise} and Upwork for workplace-sharing \cite{munoz2017mapping} are now changing the way of life in modern society. Researchers endeavor to characterize the strength and drawback of current sharing schemes and raise more efficient ones.

Numerous empirical studies on sharing have been conducted. A review of the sharing economy was presented in \cite{cheng2016sharing}. Performance of typical sharing platforms is studied in \cite{lam2018more,li2016demand} together with their impact on social welfare. The influence of sharing prices and subsidies were revealed \cite{chen2015peeking}. The key issue in energy sharing is designing a proper mechanism or business model. The desired business model should satisfy the following conditions: 1) The participant has the free choice of supply or consume; 2) The sharing market is \emph{effectively cleared}, which means supply and demand are balanced while the system constraints are not violated; 3) Being fair to all participants. Above empirical studies offer the guidance for designing sharing mechanism. The existing work can be divided into three categories according to the model structure.

\textbf{Stackelberg Game based Sharing.} This category follows a two-level (sequential) decision-making framework. In the upper level, the operator of a sharing platform announces buying and selling prices, aiming at optimizing a certain objective (e.g., profit or social welfare); In the lower level, prosumers/costumers decide the sharing profile in response to the price signal. The two levels constitute a Stackelberg game. As the operator has the priority, this scheme is called ``\emph{Operator-Dominant Sharing}". Researches in this line include optimal pricing \cite{banerjee2016multi}, efficiency in matching \cite{stiglic2015benefits} and equilibrium analysis \cite{Sharing-twosided1}, to name just a few. The pricing strategy for energy sharing among PV prosumers was investigated in \cite{Sharing-setprice1,cui2017distributed}. The existence and uniqueness of the Stackelberg Equilibrium (SE) were proved. The sharing was further categorized into direct sharing (within a single time period) and buffered sharing (across time periods) with the help of energy storage \cite{Sharing-setprice2}. A Stackelberg game between the auctioneer and the residential units was modeled in \cite{tushar2016energy} for energy storage sharing. A two-sided market based sharing pricing problem among renters and owners was studied in \cite{Sharing-twosided1}. Efficiency loss during the tradeoff between revenue and social welfare was bounded. Optimal loyalty program for sharing platform operator was investigated in \cite{fang2018loyalty}, showing that a linear loyalty program works better than the one-time sign-up bonuses. For homogeneous suppliers, a linear loyalty program fits well while for heterogeneity suppliers, a set of multi-threshold linear loyalty programs are implementation-friendly. In this category, sharing participants are price-takers, and the study focuses on the profit of the platform operator to promote the business model.

\textbf{Generalized Nash game based Sharing.} This category follows a single-level (simultaneous) decision-making framework, and the behaviors of participants are modeled as a generalized Nash game (GNG). The sellers bid the production costs while the buyers bid the demands to the market at the same time; then the sharing market is cleared and the sharing price is given. This scheme is called ``\emph{Player-Dominant Sharing}". This kind of studies has a similar structure as the traditional energy market analysis \cite{chen2018energy,chen2018optimal}, but the main difference is that the sharing participants can benefit from either sharing their resources or consuming it themselves. Vehicle-to-vehicle charge sharing was analyzed in \cite{Sharing-twosided3}. A double-sided bidding mechanism in the mobile cloud was proposed to encourage users' participation \cite{tang2017double}. In the above work, roles of the players (a seller or a buyer) are predetermined. With the prevalent of prosumers, it is natural that participants may change their roles over time. The behavior of networked prosumers with different price preferences on the trades with their neighbors was considered under peer-to-peer energy exchange \cite{cadre2018peer}. The economic efficiency of energy hubs management in a multi-energy system under three schemes (individual, sharing, aggregation) was compared in \cite{chen2018analyzing}.

\textbf{Coalition Game based Sharing.} The above two categories of researches adopt noncooperative game theoretical frameworks, while this category employs a cooperative game model. For coalition game based sharing, a set of reallocation rules is pre-determined by the sharing platform and served as prior knowledge; then each player makes his decision taking into account the reallocation he may obtain. The crucial issue in the kind of study is the proper design of the reallocation rule, such that all players are willing to cooperate, forming a coalition game. The famous Vickrey-Clarke-Groves (VCG) mechanism is one of the examples \cite{makowski1987vickrey}. An efficient cost allocation scheme for energy sharing was proposed in \cite{Sharing-reallocation2}. Coalitional game theory based local power exchange algorithm was designed for networked microgrids \cite{mei2019coalitional}. A conceptual design for the demand-side energy resource was provided in \cite{qi2017sharing}, where the coalition surplus is distributed between aggregators and prosumers. The coalition game based sharing can achieve social optimum, but for a complex sharing system, a proper reallocation rule is difficult to obtain.

This work investigates energy sharing from the perspective of prosumers via generalized Nash game, and thus falls into the second category. It tries to set forth an efficient sharing mechanism that allows participants to choose their roles flexibly, and enhances social efficiency while maintaining system security. It possesses some salient features:

1) \textbf{Sharing mechanism considering network constraints.} A generic supply-demand function based sharing mechanism is proposed. It allows prosumers to choose their roles as buyers or sellers freely. Different from previous works focusing on profit maximization, the proposed sharing scheme targets at minimizing the variance of sharing prices and maintaining fairness. When the sharing market is effectively cleared, supply-demand balance and system constraints are met. The model turns out to be a generalized Nash game. The existence and uniqueness of a generalized Nash equilibrium (GNE) in non-congestion cases is proved; an improved sharing mechanism with price regulation is derived for locating a unique GNE when congestion occurs.

2) \textbf{Desired properties of the sharing mechanism.} It is proved that every prosumer has the motivation to take part in sharing by revealing that his cost under sharing is no worse than under individual decision-making. This also means a Pareto improvement is achieved among all prosumers. The sharing price is shown to have a similar structure as the locational marginal price (LMP) in power market analysis, consisting of two parts: one is related to energy balancing and the other is caused by congestion. The proposed sharing mechanism approaches social optimum with an increasing number of participants. Even with a fixed number of resources, introducing competition can help reducing the total social cost.

3) \textbf{Deeper insights into the sharing mechanism.} Several examples are given for a better understanding of the proposed sharing mechanism. A simple case with two prosumers is used to show that the improved sharing mechanism guarantees the existence of a unique GNE. Then, the impact of flow limits is investigated. The proposed sharing mechanism always results in a fair nodal sharing price. It is also revealed that getting more prosumers involved and introducing competition help reduce the total cost. One important factor regarding competition is the location of lines.

The rest of this paper is organized as follows. The mathematical formulations of energy prosumers and description of the intuitive energy sharing mechanism are presented in Section II; Discussions regarding the existence of GNE are given in Section III; An improved sharing mechanism is provided in Section IV to ensure the uniqueness of GNE. Properties of the improved sharing game are revealed in Section V; Some illustrative examples are given in Section VI. Finally, conclusions are summarized in Section VII.

\section{Game Model of Energy Sharing}
\subsection{Energy Prosumers}

A set of prosumers indexed by $i \in \mathcal{I}=\{1,2,...,I\}$ is considered and each of them possesses $K$ resources, i.e., a distributed generator (DG), indexed by $k \in \mathcal{K}=\{1,2,...,K\}$. Let the underlying simple undirected graph be $G=(V,E)$, where $V$ denotes the set of buses and $E$ the set of edges. The degree of a vertex of graph $G=(V,E)$ is defined as the number of edges incident to that vertex. Different from traditional producers or consumers, prosumers can both produce and consume energy. The self-balancing condition of prosumer $i$ is shown below.
\bq
\sum \limits_k p_i^{k0}+E_i^0=D_i^0
\eq  
where $\sum_k p_i^{k0}$ is the energy produced by itself, $E_i^0$ is the energy purchased from the grid, and $D_i^0$ is the load demand it needs to satisfy, which is fixed. We consider all these prosumers taking part in a demand response program and the required load reduction for prosumer $i$ is $D_i$, which means $E_i^0$ should be reduced by $D_i$. If each prosumer makes the decision individually, in other words, it can only adjust its resources output by $p_i^k,\forall k$ to meet the load reduction requirement, we have $\sum_k p_i^k=D_i$. The quadratic function $f_i(p_i)=\sum_k c_i^k(p_i^k)^2$ is adopted to depict the disutility of prosumer $i$. According to Cauchy-Schwarz inequality, the disutility is lower bounded by
\bq
\sum \limits_k c_i^k (p_i^k)^2 \ge \frac{D_i^2}{\sum \limits_k \frac{1}{c_i^k}} \nonumber
\eq

As the disutility varies among different prosumers, the individual decision-making scheme may not be the most efficient one, which means the social total disutility can still be diminished. An intuitive idea is to allow them to trade with each other so that the prosumer with lower marginal disutility can reduce more and buy energy from prosumer with higher marginal disutility, resulting in a win-win game. Regarding energy sharing, there are two main concerns: (1) a proper profit allocation scheme should be designed so that every prosumer is willing to participate; (2) deploying energy transaction might violate energy flow limits of the network. To this end, an effective sharing mechanism that can motivate prosumers to get involved while maintaining system security is necessary.

\begin{assumption}
\label{assumption-1}
 The maximum degree of vertexes of graph $G=(V,E)$ is finite, and without loss of generality, assume the upper bound is $G^D$.
\end{assumption}

\subsection{Generic Supply-Demand Function}
A generic supply-demand function is adopted, allowing prosumers to flexibly alter his role either as a seller or a buyer. The supply-demand function for prosumer $i$ is
\bq 
q_i= -a \lambda_i+b_i 
\eq
where $q_i$ is the clearing amount of prosumer $i$ ($q_i>0$ for buyers and $q_i<0$ for sellers); $\frac{1}{a}>0$ is the price sensitivity of the sharing market, which shows how the bids of prosumers influence the sharing price; $\lambda_i$ is the clearing price for prosumer $i$; $b_i$ is the bid of prosumer $i$. The sharing market is said to be cleared effectively when the following two conditions are met:

(1) The net quantity $\sum_i q_i=0$, which means
\bq
\label{eq:clearing}
\sum \limits_{i \in \mathcal{I}} (-a\lambda_i+b_i)=0
\eq

(2) The energy transaction pattern corresponding to the market clearance is deliverable, which means there exists a feasible power flow solution given the offers and demands of prosumers. Here, direct current (DC) model is used to calculate the power flow in each line $l \in \mathcal{L}=\{1,2,...,L\}$. Since power balancing is ensured by (\ref{eq:clearing}), we only need to impose the following flow limit constraints
\bq
\label{eq:flowlimits}
-F_l \le \sum \limits_{i \in \mathcal{I}} \pi_{il} (-a\lambda_i+b_i) \le F_l,\forall l \in \mathcal{L}
\eq  

The sharing platform clears the market subject to constraints \eqref{eq:clearing} and \eqref{eq:flowlimits}. In view of possible congestions, the sharing price $\lambda_i$ may vary at different buses. To give a fair result, the objective of the market clearance problem is to minimize the variance of the sharing prices $\lambda_i,\forall i$. 
\begin{equation}
\label{eq:price-var}
\min~ \frac{1}{I} \sum_i \left(\lambda_i - \frac{\sum\nolimits_j \lambda_j}{I}\right)^2
\end{equation}
In consideration of (\ref{eq:clearing}), by ignoring the constant terms, the objective function can be simplified into $\sum_i \lambda_i^2/I$.

\subsection{Energy Sharing Mechanism}
Under the settings described above, the energy sharing market considering network constraints is operated following three steps and shown in Fig. \ref{fig:procedure}.
\begin{figure}[!t]
\centerline{\includegraphics[width=0.6\columnwidth]{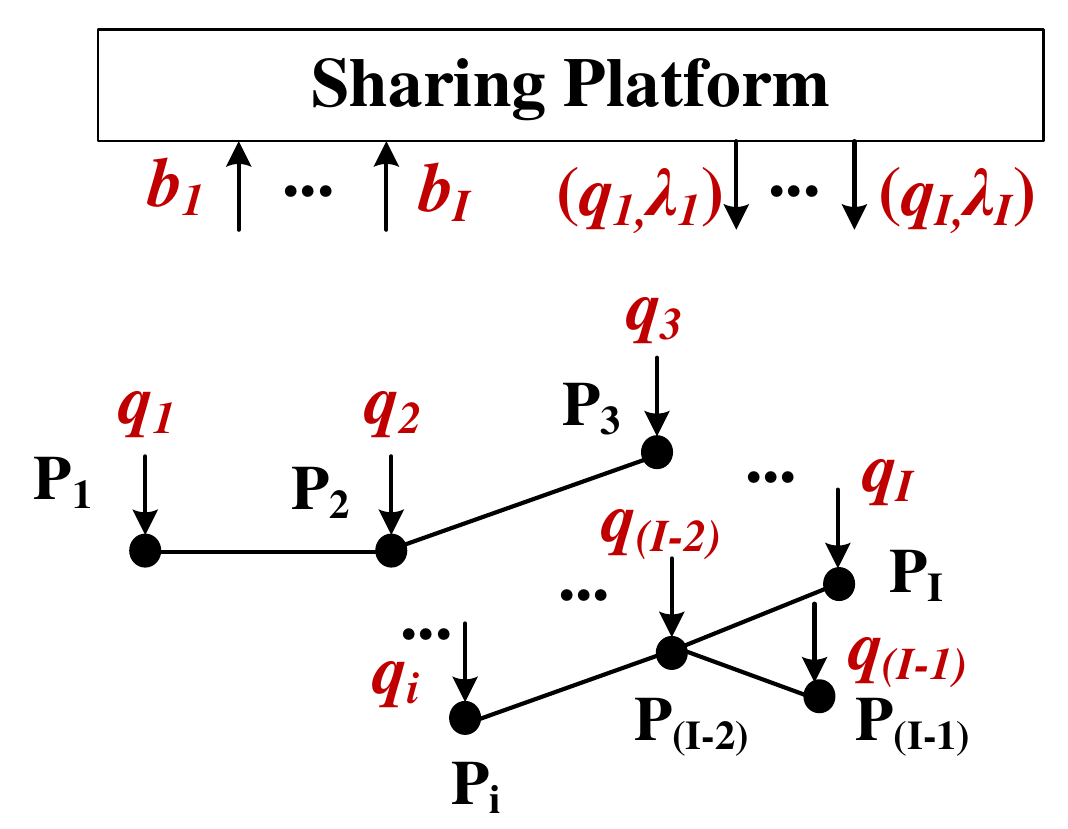}}
\setlength{\abovecaptionskip}{0pt}
\caption{Procedure of the sharing mechanism.}
\setlength{\belowcaptionskip}{-5em}
\label{fig:procedure}
\vspace{-1em}
\end{figure}

\textbf{Step 1:} Estimate the value of price sensitivity factor $a$ via historical data. Each prosumer submits a bid $b_i$ to the sharing platform, showing his willingness to buy.

\textbf{Step 2:} The sharing platform effectively clears the market by solving problem \eqref{eq:platform}. Then the sharing price $\lambda_i$ and clearing amount $q_i$ are sent back to prosumer $i$.
\bsq
\label{eq:platform}
\bq
\mathop {\min }\limits_{\lambda_i,\forall i} && \frac{1}{I}\sum \limits_i \lambda_i^2  \\
\rm{s.t.}&& \sum \limits_{i} (-a\lambda_i+b_i)=0  : \eta \\
&& -F_l \le \sum \limits_{i} \pi_{il} (-a\lambda_i+b_i) \le F_l: 
\alpha_l^{-},\alpha_l^{+}, \forall l
\eq
\esq
where $\eta$, $\alpha_l^{-}$, and $\alpha_l^{+}$ following a colon are dual variables. 
 
\textbf{Step 3:} The clearing amount of each prosumer is given by $q_i(b)=-a\lambda_i(b)+b_i$. If $q_i(b)>0$, he purchases energy from the sharing market and the payment is $\lambda_iq_i(b)=\lambda_i(-a\lambda_i+b_i)$; otherwise, if $q_i(b)<0$, he sells energy to the sharing market and the revenue is $-\lambda_iq_i(b)=-\lambda_i(-a\lambda_i+b_i)$. 

The sharing price $\lambda_i,\forall i$ is not neccessary to be positive. A intuitive example is that suppose all $c_i^k,\forall i,k$ are equal and $D_i<0,\forall i$, which means the prosumers are required to produce less (or consume more). If there is no congestion, it is easy to prove that $\lambda_i<0,\forall i$. At this time, a positive $q_i$ can be regarded as selling a positive load to other prosumer, and it is reasonable that it gets money from the platform.

\subsection{Energy Sharing as A Generalized Nash Game}
Given the rule of market clearing, a prosumer can foresee the clearing result and determine the optimal bidding strategy. In particular, the problem of prosumer $i \in \mathcal{I}$ becomes
\begin{equation}
\label{eq:sharing-upper}
\mathop {\min} \limits_{p_i^k,\forall k,b_i} \left\{  \Pi_i(p_i,b_i,b_{-i}) \middle| \sum \limits_k p_i^k-a\lambda_i(b)+b_i=D_i :\beta_i \right\} 
\end{equation}
where the objective function
\begin{equation}
\Pi_i(p_i,b_i,b_{-i}) = \sum \limits_k c_i^k (p_i^k)^2+\lambda_i(b)(-a\lambda_i(b)+b_i) \label{eq:Prosumer-Obj} 
\end{equation}
consisting of the disutility $f_i(p_i):=\sum_k c_i^k(p_i^k)^2$ and the sharing cost $s_i(b):=(-a\lambda_i+b_i)\lambda_i$. The constraint in \eqref{eq:sharing-upper} is the energy balancing condition and $\beta_i$ is the corresponding dual variable. $\lambda_i$ is determined from the sharing market clearing problem \eqref{eq:platform}. 

In summary, the game among all prosumers consists of the following elements:

 1) the set of prosumers $\mathcal{I}=\{1,2,...,I\}$; 

 2) action sets $X_i(b_{-i})$\footnote{The subscribe $-i$ means all players in $\mathcal{I}$ except $i$ },$\forall i$, and strategy space  $X = \prod_i X_i$; 

3) Payoff functions $\Pi_i(p_i(b), b_i, b_{-i}), \forall i$. 

For simplicity, denote by $\mathcal{G}=\{\mathcal{I},X,\Pi\}$ the abstract form of the sharing game \eqref{eq:platform}-\eqref{eq:sharing-upper}. Because the price $\lambda_i$, which appears in the strategy set $X_i$, depends on the joint action $b_{-i}$ of other prosumers, the strategy sets $X_i(b_{-i})$, $i=\{1,\cdots I\}$ of individual players are correlated. Therefore, the game $\mathcal{G}=\{\mathcal{I},X,\Pi\}$ belongs to the category of GNG. This is different from a standard Nash game in which correlation only appears in the payoff functions.

\section{  Equilibrium  of the Intuitive Sharing Game}
\subsection{Generalized Nash Equilibrium}
In this subsection, the properties of the equilibrium of the proposed sharing market are revealed. First, the definition of a GNE is given as follows.

\begin{definition}
A strategy profile\footnote{Given a collection of $x_i$ for $i$ in a certain set $A$, $x$ denotes the vector $x := (x_i; i\in A)$ of a proper dimension with $x_i$ as its components} $ b^* \in X $ is a Generalized Nash Equilibrium (GNE) of the sharing game $\mathcal{G}=\{\mathcal{I},X,\Pi\}$ defined by  \eqref{eq:platform}-\eqref{eq:sharing-upper}, if 
\bq
\Pi_i(p_i^*(b^*), b_i^*, b_{-i}^*) \le \Pi_i(p_i(b_i,b_{-i}^*), b_i, b_{-i}^*), \forall b_i \in X_i(b_{-i}^*), \forall i \nonumber
\eq
\end{definition}

Given the energy production schedule $p$, define 
\begin{equation*}
\tilde \lambda_i(p_i):=2c_i^kp_i^k+\frac{\sum_k p_i^k-D_i}{a(I-1)}
\end{equation*}
and 
\begin{equation*}
\tilde b_i(p_i):=D_i-\sum_k p_i^k+a\tilde \lambda_i(p_i)
\end{equation*}
we have the following proposition.
\begin{proposition}
\label{Thm:prop-1}
If the sharing game $\mathcal{G}=\{\mathcal{I},X,\Pi\}$ has an isolated GNE (IGNE, which means that no other GNE exists in a small enough neighborhood of the given GNE), then it is also unique. Denote by $ b^*$ the unique IGNE and $p^*$ the corresponding optimal adjustment, then $b_i^*=\tilde b_i(p_i), \forall i \in \mathcal{I}$, and $p=[p_i^*], \forall i$ is the unique solution of the following optimization problem:
\bsq
\label{eq:central}
\bq
\mathop {\min} \limits_{p_i^k,\forall i,k} && \sum \limits_i \sum \limits_k c_i^k(p_i^k)^2+\sum \limits_i \frac{(D_i-\sum_k p_i^k)^2}{2a(I-1)} \label{eq:central.1}\\
\rm{s.t.} && \sum \limits_i \sum \limits_k p_i^k =\sum \limits_i D_i : \kappa \label{eq:central.2}\\
&& -F_l \le \sum \limits_i \pi_{il}(D_i-\sum \limits_k p_i^k) \le F_l:\tau_l^{-},\tau_l^{+} \label{eq:central.3}
\eq
\esq
\end{proposition}

The proof are given in Appendix A. Proposition \ref{Thm:prop-1} offers a convenient way to identify the IGNE, if one does exist. In what follows, we discuss the existence of the IGNE. 

\subsection{Existence of an Isolated Generalized Nash Equilibrium}

Unlike a standard Nash game whose equilibrium exists and is unique under certain convexity assumptions, a GNG may possess infinitely much equilibrium which are non-isolated, because the flow limit constraints complicate the problem. A simple example is given below. 

\textbf{Example:} Suppose there are two prosumers connecting to the head bus and the tail bus of a line; each of them controls one resource. Let $a=1$, and then the market clearing problem is
\bsq
\bq
\mathop{\min} && \frac{1}{2}[\lambda_1^2+\lambda_2^2] \\
\rm{s.t.} && \lambda_1+\lambda_2-b_1-b_2=0 \\
&& -F_1 \le  \lambda_1-b_1 \le F_1 
\eq
\esq
Fix the strategy of prosumer 2, the optimal solution $\lambda_1$ is a function of $b_1$, which is
\bq
\label{eq:lambda_1}
\lambda _1^* = \begin{cases} 
F_1+b_1, & {\rm{if}} ~ \frac{b_1+b_2}{2}\ge F_1+b_1 \\
\frac{b_1+b_2}{2}, & {\rm{if}} ~ -F_1+b_1 \le \frac{b_1+b_2}{2} \le F_1+b_1 \\
-F_1+b_1,& {\rm{if}} ~ \frac{b_1+b_2}{2}\le -F_1+b_1 
\end{cases}
\eq
The first/last one corresponds to the case in which delivered flow reaches lower/upper bound, and flow constraint is inactive in the second one. Then, we solve the equivalent problem \eqref{eq:central} considering the following situations:

(1) The network is not congested (constraint \eqref{eq:central.3} does not influence the optimal solution). The KKT conditions of \eqref{eq:central} stipulate
\begin{gather*}
\lambda_1=\lambda_2=2c_1p_1+p_1-D_1=2c_2p_2+p_2-D_2  \\
b_1=D_1-p_1+\lambda_1, \;\; b_2=D_2-p_2+\lambda_2 
\end{gather*}
It is easy to check that $\lambda_1=(b_1+b_2)/2$, which corresponds to the second situation of \eqref{eq:lambda_1}. As a result, the equilibrium is uniquely determined. Actually, when there is no congestion, the sharing market degenerates to the case without network constraints in \cite{yue2019arxiv}, where the existence and uniqueness of a NE have been proved.

(2) Delivered flow reaches the lower/upper bound, e.g., $p_1^*-D_1=F_1$ or $p_1^*=D_1+F_1$, then $p_2^*=-F_1+D_2$. The optimal strategy for prosumer 1 is to choose $b_1^*$ that satisfies $(b_1^*+b_2^*)/2=F_1+b_1^*$, which means $b_1^*=b_2^*-2F_1$.
In addition, we can figure out the following relations from the KKT condition
\begin{gather*}
b_1^*=2c_1p_1^* = 2c_1(D_1+F_1) \\
b_2^*=2c_2p_2^* = 2c_2(D_2-F_1)
\end{gather*}
Suppose an IGNE exists, we must have $\lambda_1=\lambda_2$, which means there is no congestion. As a result, there is no IGNE.

But it is worth noting that Any $(p^*,b^*)$ that satisfies $b_1^*=-2F+b_2^*$, $p_1^*=D_1+F_1$ and $p_2^*=D_2+F_2$ are GNEs of this sharing game.
This observation can be generalized in more general cases.
\begin{proposition}
\label{Thm:prop-2-2}
If congestion constraint is redundant, then an IGNE for the sharing game   $\mathcal{G}=\{\mathcal{I},X,\Pi\}$ exists; otherwise, there is none or infinitely many non-isolated GNEs.
\end{proposition}
Proposition \ref{Thm:prop-2-2} is provided in Appendix C, following the proof of Proposition  \ref{Thm:prop-2}. It confirms a negative conclusion that when the network constraints are binding, there is no IGNE. The reason is the market power of prosumers. Nonetheless, the existence of IGNE can be retrieved via a minor modification on the objective function to restrict the market power. 

\section{Improved  Sharing Mechanism}
Before getting into details of the improved energy sharing mechanism, we first give the following lemma.
\begin{lemma}
The optimal output $p_i^{k*},\forall k$ of prosumer $i$ satisfies
\bq
2c_i^kp_i^{k*}=\frac{2\sum_k p_i^{k*}}{\sum_k \frac{1}{c_i^k}}:=md_i
\eq
\end{lemma}
\begin{proof}
According to Cauchy-Schwarz Inequality, we have
\bq
\sum \limits_k c_i^k (p_i^k)^2 \ge \frac{1}{\sum_k \frac{1}{c_i^k}}(D_i+a\lambda_i-b_i)^2 \nonumber
\eq
The equality holds when and only when 
$$c_i^1p_i^1=c_i^2p_i^2=...=(\sum_k p_i^k)/(\sum_k \frac{1}{c_i^k}):=\bar c_i \sum_k p_i^k$$. This completes the proof.
\end{proof}

To restrain prosumers' market power, a price regulation procedure is included in \textbf{Step 2} of the sharing procedure in Section II.C. The ultimate sharing price sent to prosumer $i$ is
\bq
\label{eq:priceregulate}
\lambda_i^c = \begin{cases}
\max \left\{ \lambda_i, md_i-\frac{q_i}{a(I-1)} \right\}, & q_i \ge 0 \\
\min \left\{ \lambda_i, md_i-\frac{q_i}{a(I-1)} \right\}, & q_i < 0
\end{cases}
\eq
It means that the difference between the sharing price $\lambda_i^c$ for prosumer $i$ and his marginal disutility $md_i$ is restricted. To be specific, when prosumer $i$ buys from the sharing market ($q_i=-a\lambda_i+b_i>0$), the price privilege ($md_i-\lambda_i^c$) he can get from sharing is at most $(D_i-\sum_k p_i^k)/{a(I-1)}$; when he sells to the sharing market ($q_i=-a\lambda_i+b_i<0$), the price privilege ($\lambda_i^c-md_i$) he can get from sharing is at most $(\sum_k p_i^k-D_i)/a(I-1)$. The maximum price privilege of prosumer $i$ is related to the degree of sharing participation reflected by the sharing amount $q_i$, price sensitivity factor $a$ and the number of prosumers $I$. 

Under this revamped sharing mechanism, the sharing problem of prosumer $i$ becomes
\bsq
\label{eq:modified-obj}
\bq
\mathop {\min} \limits_{p_i^k,\forall k,b_i} && \sum \limits_k c_i^k (p_i^k)^2+u_i(\lambda_i,b_i,p_i) \\
\rm{s.t.} && \sum \limits_k p_i^k-a\lambda_i(b)+b_i=D_i
\eq
\esq
where
\bq
\label{eq:price-regulate}
u_i(\lambda_i,b_i,p_i)  = 
\max (\lambda _i q_i,(md_i - \frac{q_i}{a(I - 1)})q_i)
\eq
The total cost is redefined as 
\begin{equation*}
\Gamma_i(\lambda_i,b_i,p_i):=\sum_k c_i^k (p_i^k)^2+u_i(\lambda_i,b_i,p_i)
\end{equation*} 
and $\lambda_i,\forall i$ is the solution of problem \eqref{eq:platform}. With this improved sharing mechanism, the existence and uniqueness of GNE are guaranteed by the following proposition.


\begin{proposition}
\label{Thm:prop-2}
The sharing game \eqref{eq:modified-obj} with \eqref{eq:platform} has a unique GNE $\hat b$ and $\hat p$ the corresponding optimal adjustment, where $\hat{b}_i=\tilde b_i(\hat{p}_i)$, $\forall i \in \mathcal{I}$, and $\hat{p}_i$ is the unique solution of
\bsq
\label{eq:central2}
\bq
\mathop {\min} \limits_{p_i^k,\forall i,k} && \sum \limits_i \sum \limits_k c_i^k(p_i^k)^2+\sum \limits_i \frac{(D_i-\sum \nolimits_k p_i^k)^2}{2a(I-1)} \label{eq:central2.1}\\
\rm{s.t.} && \sum \limits_i \sum \limits_k p_i^k =\sum \limits_i D_i : \kappa \label{eq:central2.2}\\
&& -F_l \le \sum \limits_i \pi_{il}(D_i-\sum \limits_k p_i^k) \le F_l:\tau_l^{-},\tau_l^{+} \label{eq:central2.3}
\eq
\esq
\end{proposition}

The proof can be found in Appendix B. With the price regulation policy (\ref{eq:priceregulate}),  the market equilibrium is well defined and inherits the properties of the original sharing game.


\section{Properties of the Improved Sharing Mechanism}
\subsection{Individual Rationality of Prosumers}
After characterizing the GNE, now we can show that every prosumer has the incentive to share by comparing the costs associated with the individual decision-making and the GNE. Let $\Pi_i(p^*_i,b^*)$ be the cost of prosumer $i$ at the GNE of sharing game $\mathcal{G}=\{\mathcal{I}, X,\Pi\}$ defined by \eqref{eq:platform} and \eqref{eq:modified-obj}, and $f_i(D_i)$ the cost of prosumer $i$ under individual decision-making.
\begin{proposition}
\label{Thm:prop-3} The following relation holds
\bq
\label{eq:pareto}
\Pi_i (p^*_i,b^*) \le f_i(D_i), \forall i\in \mathcal{I} 
\eq
\end{proposition}

The proof can be found in Appendix D. It guarantees that with the proposed mechanism, every prosumer can benefit from sharing, such that a Pareto improvement can be achieved. This satisfies one of the two main concerns for sharing mechanism design stated in Section II.A.

\subsection{Sharing Price}

In this subsection, the relationship between the sharing price $\lambda_i^c$ that clears the market, prosumer's marginal disutility, and line congestion is revealed.
 
\begin{proposition}
\label{Thm:prop-4}
Assume $b^*$ is the GNE of the sharing game \eqref{eq:platform} and \eqref{eq:modified-obj}. Then the  sharing price at market equilibrium is given by
\bq
\lambda_i^{c*} = \lambda_i^* = -\kappa^*-\sum_l \pi_{il}\tau_{l}^{*-}+\sum_l \pi_{il}\tau_{l}^{*+}\nonumber
\eq
\end{proposition}

The proof can be found in Appendix E. We can observe that the sharing price for prosumer $i$ consists of two parts: the first term $-\kappa^*$ corresponds to the marginal cost of energy and  is the same for all prosumers; the remaining terms $-\sum_l \pi_{il}\tau_{l}^{*-}+\sum_l \pi_{il}\tau_{l}^{*+}$ reflect prosumer $i$'s contribution on line congestion. Such a price exhibits a structure similar to the locational marginal price \cite{bose2019some}.

\textbf{Discussion:} Another important issue is whether a subsidy is needed to run the platform. The sum of the sharing costs $s_i(b^*)$ for all prosumers at the equilibrium state is
\bq
&& \sum\limits_i \lambda_i^{c*}(-a\lambda_i^*+b_i^*)= \sum\limits_i \lambda_i^{*}(-a\lambda_i^*+b_i^*) \nonumber\\
&=& \sum \limits_i (-\kappa^*-\sum \limits_l \pi_{il}\tau_l^{*-}+\sum \limits_l \pi_{il} \tau_l^{*+})(-a\lambda_i^*+b_i^*) \nonumber\\
& = & \sum \limits_i [-\sum \limits_l \pi_{il}(-a\lambda_i^*+b_i^*)\tau_l^{*-}+\sum \limits_l \pi_{il} (-a\lambda_i^*+b_i^*)\tau_l^{*+}] \nonumber\\
& = & \sum \limits_l [-\sum \limits_i \pi_{il}(-a\lambda_i^*+b_i^*)\tau_l^{*-}+\sum \limits_i \pi_{il} (-a\lambda_i^*+b_i^*)\tau_l^{*+}] \nonumber\\
& = & \sum \limits_l F_l(\tau_l^{*-}+\tau_l^{*+}) \ge 0
\eq
If there is no congestion in the network, then the sum of $s_i(b^*)$ is equal to zero, which means that the sharing market is self-balanced economically. In the presence of congestion, after sharing, not only every prosumer is better-off and a positive payment is given to the sharing platform, which is similar to the renown concept of financial transmission right (FTR) \cite{bose2019some}. 

\subsection{Social Efficiency}
According to Proposition \ref{Thm:prop-3}, the proposed sharing mechanism can improve the efficiency compared with individual decision-making. In this subsection, we show that how close is the efficiency of the proposed mechanism compared with the centralized decision making problem:
\bsq
\label{eq:AGG}
\bq
\mathop {\min }\limits_{{p_{i,\forall i \in \mathcal{I}}}} && \sum\limits_i \sum \limits_k {{c_i^k}(p_i^k)^2} \label{AGG.1}\\
{\rm{s}}{\rm{.t}}{\rm{.}} && \sum\limits_i \sum \limits_k {{p_i^k}}  = \sum\limits_i {{D_i}} \label{AGG.2}\\
&& -F_l \le \sum \limits_i \pi_{il}(D_i-\sum \limits_k p_i^k) \le F_l
\eq
\esq 

\begin{definition} (Socially Optimum)
$\bar p$ is socially optimal if $\bar p$ is the unique optimal solution of \eqref{eq:AGG}.
\end{definition}

Proposition \ref{Thm:prop-2} claims that the GNE of the modified sharing game always exists and can be extracted from the optimal solution of problem \eqref{eq:central2}. Problems \eqref{eq:central2} and \eqref{eq:AGG} seem alike except for the second term $\sum_i (D_i-\sum_k p_i^k)^2/2a(I-1)$ in the objective function. This term approaches zero when $I \rightarrow \infty$, and thus problem \eqref{eq:central2} is identical to the social optimal problem \eqref{eq:AGG}. The gap between sharing and the social optimum when $I \rightarrow \infty$ is stated in the following proposition.

\begin{proposition}
\label{Thm:prop-5}
Let $b^*(I)$ be the GNE of \eqref{eq:platform},\eqref{eq:modified-obj}, $p^*(I)$ the corresponding optimal adjustment and $\bar p(I)$ be the socially optimal solution of \eqref{eq:AGG}. Then, we have
$$\sum \nolimits_{i \in \mathcal{I}} f_i(p_i^*(I)) \ge \sum \nolimits_{i \in \mathcal{I}} f_i(\bar p_i(I))$$
and the average cost difference
\bq
\mathop {\lim }\limits_{I \to \infty } \frac{1}{I}\left[ {\sum\nolimits_{i \in \mathcal{I}} {{f_i}(p_i^*(I)) - \sum\nolimits_{i \in \mathcal{I}} {{f_i}({{\bar p}_i}(I))} } } \right] = 0 \nonumber
\eq
\end{proposition}

The  proof can be found in Appendix F. It declares that the proposed sharing mechanism can approach the social optimum with an increasing number of market participants.

\subsection{Role of Competition}
We have shown that social efficiency can be enhanced by employing more prosumers, and thus more resources, without changing the resource endowment of original prosumers. In this part, we consider another situation that the total number of resource is fixed and show that introducing competition can benefit social efficiency.

Let $(I,K,D)$ denotes a scenario that there are $I$ prosumers and each of them possesses $K$ resources, the demand adjustment for them is $D:=\{D_i,\forall i\}$. Assume that the GNE under scenario $(I,K,D)$ is $b^*$ and $p^*$ the corresponding optimal adjustment. We introduce competition in the following way:

The resource prosumer $i$ owns is equally distributed to $M$ prosumers ($K$ is divisible by $M$) and the demand adjustment accordingly, and $D_i^m,\forall m$ satisfies
\bq
 \sum \limits_{j=\frac{(m-1)K}{M}+1}^{\frac{mK}{M}} p_{i}^{j*}- D_{i}^m & = & \frac{1}{M} (\sum \limits_k p_i^{k*}- D_i), m=\{1,2,...,M\},\forall i \nonumber
\eq

After introducing competition, $I^{'}=MI$, $K ^{'}=K/M$ and $D^{'}:=\{D_i^m,\forall i,m\}$. We call the new scenario $(I^{'},K^{'},D^{'})$ a equal partition of scenario $(I,K,D)$. With this equal partition rule, we have the following proposition.
\begin{proposition}
\label{Thm:prop-6}
Suppose $(p^*(I,K,D),b^*(I,K,D))$ is the unique GNE of the sharing problem \eqref{eq:modified-obj} and \eqref{eq:platform} with $I>1$, then for scenario $(I,K,D)$ and its equal partition $(I^{'},K^{'},D^{'})$, we have
\bq
\sum \limits_{i=1}^{I} f_i(p_i^*(I,K,D)) \ge \sum \limits_{i=1}^{I^{'}} f_i(p_i^*(I^{'},K^{'},D^{'}))
\eq
\end{proposition}

The proof can be found in Appendix G. It demonstrates that the system is the most efficient when all resources are possessed by one prosumer; otherwise, introducing competition in an equal partition way can reduce the total social cost and thus enhance the social efficiency.

\section{Illustrative Examples}

In this section, numerical experiments are presented to illustrate theoretical results. First, a simple case is used to show the basic setup. Then, the impacts of several factors are analyzed, including the impact of network constraints, number of prosumers as well as the role of competition.
 
\subsection{Benchmark Case}
The simplest case with only two prosumers is taken as an illustrative example. Each of them owns one resource and is located at one side of a line. $p_1/p_2$ is the resource output of prosumer 1/2. The price sensitivity factor is $a=1$ and the cost coefficients are $c_1=2.5$, $c_2=3.5$. The required demand adjustments are $D_1=3$ and $D_2=7$. First we let $F_1=10$ which is the case without congestion, and the best response curves of two prosumers are plotted at the left side of Fig. \ref{fig:BRC}. A unique generalized Nash equilibrium $(b_1^*,b_2^*)=(27.14,32.00)$ and $(p_1^*,p_2^*)=(5.43,4.57)$ is given by the intersection of two curves. Then we decrease the flow limit and choose $F_1=2$, creating a case with congestion. With the intuitive sharing mechanism in Section II, the best response curves are shown in the middle of Fig. \ref{fig:BRC}. Two curves coincide meaning that there are infinity many GNEs but no IGNE. With the improved sharing mechanism in Section IV, the best response curves are drawn in the right-hand side of Fig. \ref{fig:BRC}, giving an unique GNE $(b_1^*,b_2^*)=(25.00,35.00)$ and $(p_1^*,p_2^*)=(5,25,5,35)$.
\begin{figure}[!t]
\centerline{\includegraphics[width=1.05\columnwidth]{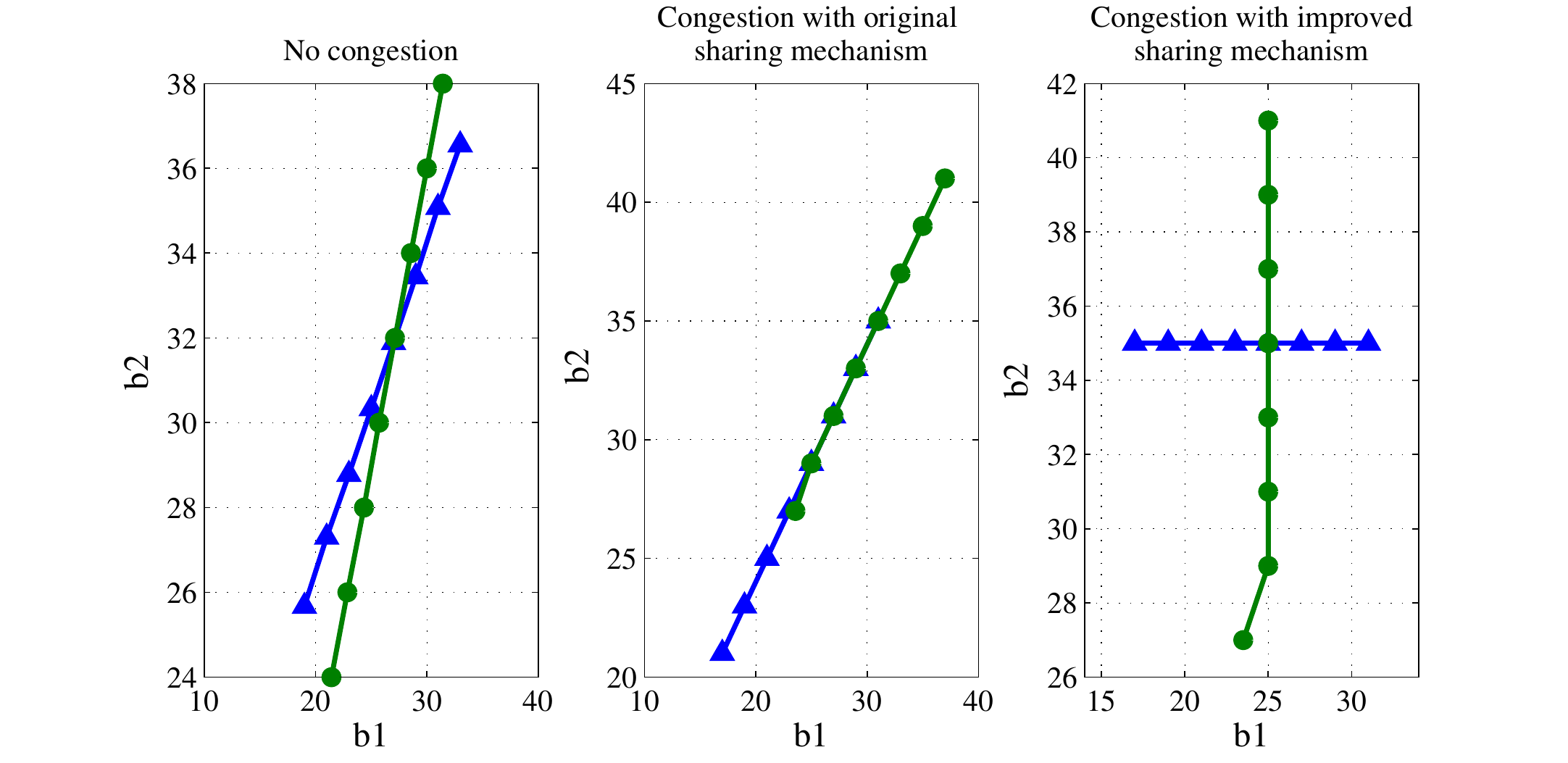}}
\setlength{\abovecaptionskip}{0pt}
\caption{Best response curves in different situations.}
\setlength{\belowcaptionskip}{-5em}
\label{fig:BRC}
\vspace{-1em}
\end{figure}

\subsection{Impact of Flow Limit}
In this section, the impact of flow limit is investigated. A simple network with three prosumers, each at one vertex of a triangle, is tested with $c=[2.5, 3.5, 4.5]$ and $D=[3,7,11]$. We change the flow limit of all lines simultaneously from 1 to 3.5, the change of nodal prices under social optimal (SCO) and sharing market equilibrium (SMK) as well as the social total costs are shown in Fig. \ref{fig:FLM}. Both social total costs decrease with relaxing flow limit, and their relative differences are all less than 0.008\%, showing that the proposed sharing mechanism can achieve a near-optimal solution, which is efficient. The variances of nodal prices under SCO and SMK decline with a looser restriction on energy flows. Under a specific flow limit, the variance of nodal prices under SMK is smaller than that under SCO, reducing the price discrimination among prosumers.
\begin{figure}[!t]
\centerline{\includegraphics[width=0.8\columnwidth]{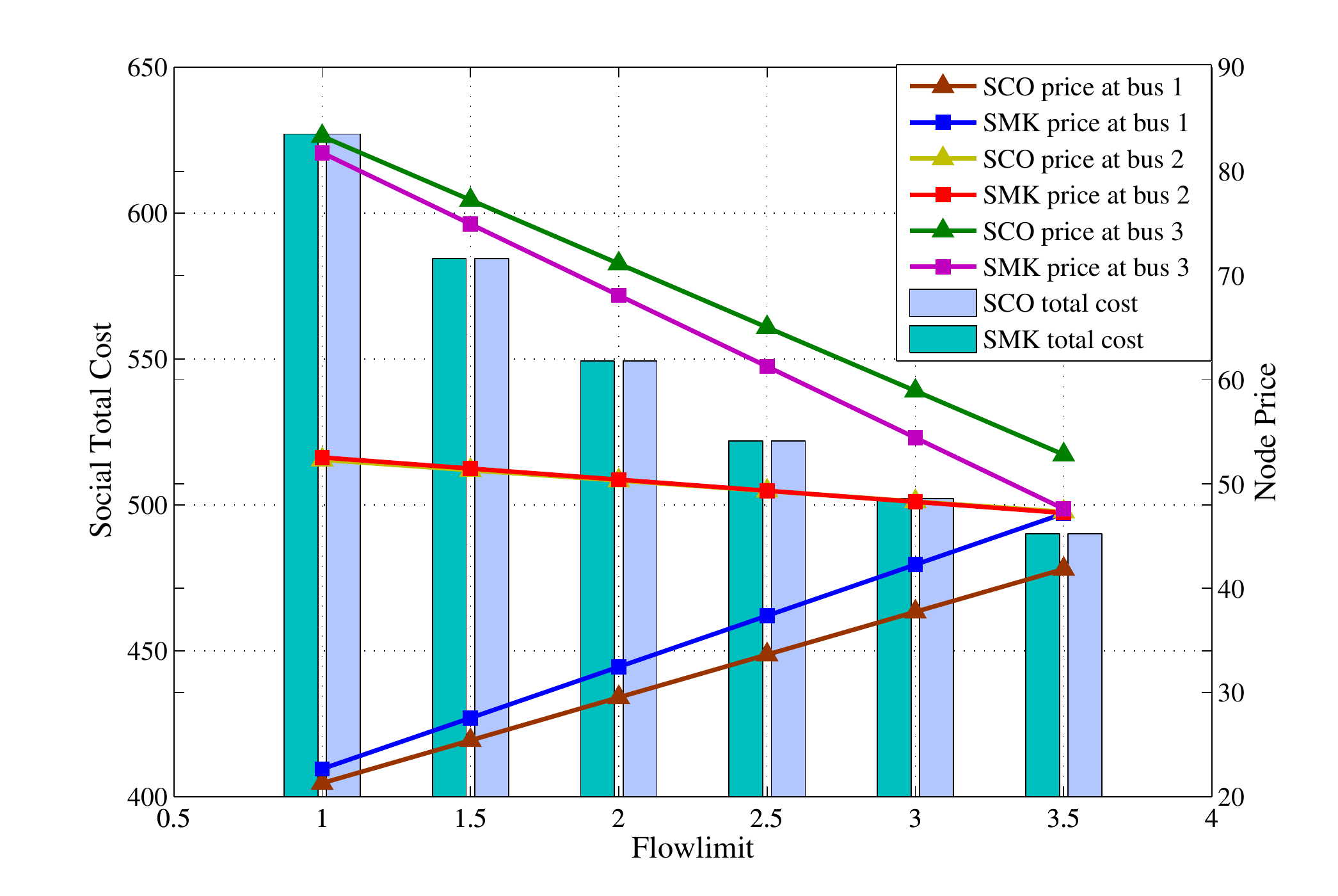}}
\setlength{\abovecaptionskip}{0pt}
\caption{Nodal prices and social total costs under SCO and SMK.}
\setlength{\belowcaptionskip}{-5em}
\label{fig:FLM}
\vspace{-1em}
\end{figure}
\subsection{Impact of the Number of Prosumers}
We change the number of prosumers $I$ from 2 to 30. The parameters $c_i$, $D_i$ are randomly chosen and 10 scenarios are tested. The topology of the test system is shown in Fig.\ref{fig:topology}. Average cost gaps between SCO and SMK with different $I$ are given in Fig. \ref{fig:ARD}. We can observe that the average total cost under sharing is always larger than that under social optimum, but the gap converges to zero with the increase of $I$.
\begin{figure}[!t]
\centerline{\includegraphics[width=0.7\columnwidth]{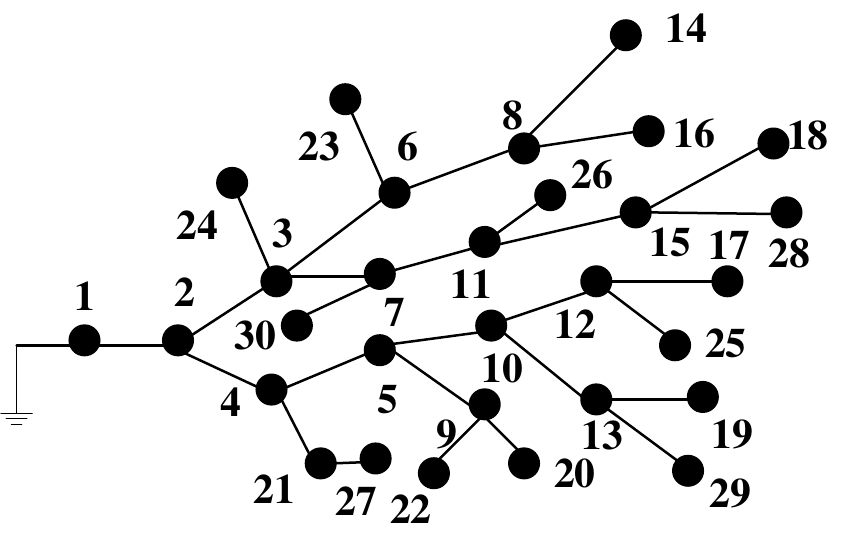}}
\setlength{\abovecaptionskip}{0pt}
\caption{Topology of the test system.}
\setlength{\belowcaptionskip}{-5em}
\label{fig:topology}
\vspace{-1em}
\end{figure}
\begin{figure}[!t]
\centerline{\includegraphics[width=0.85\columnwidth]{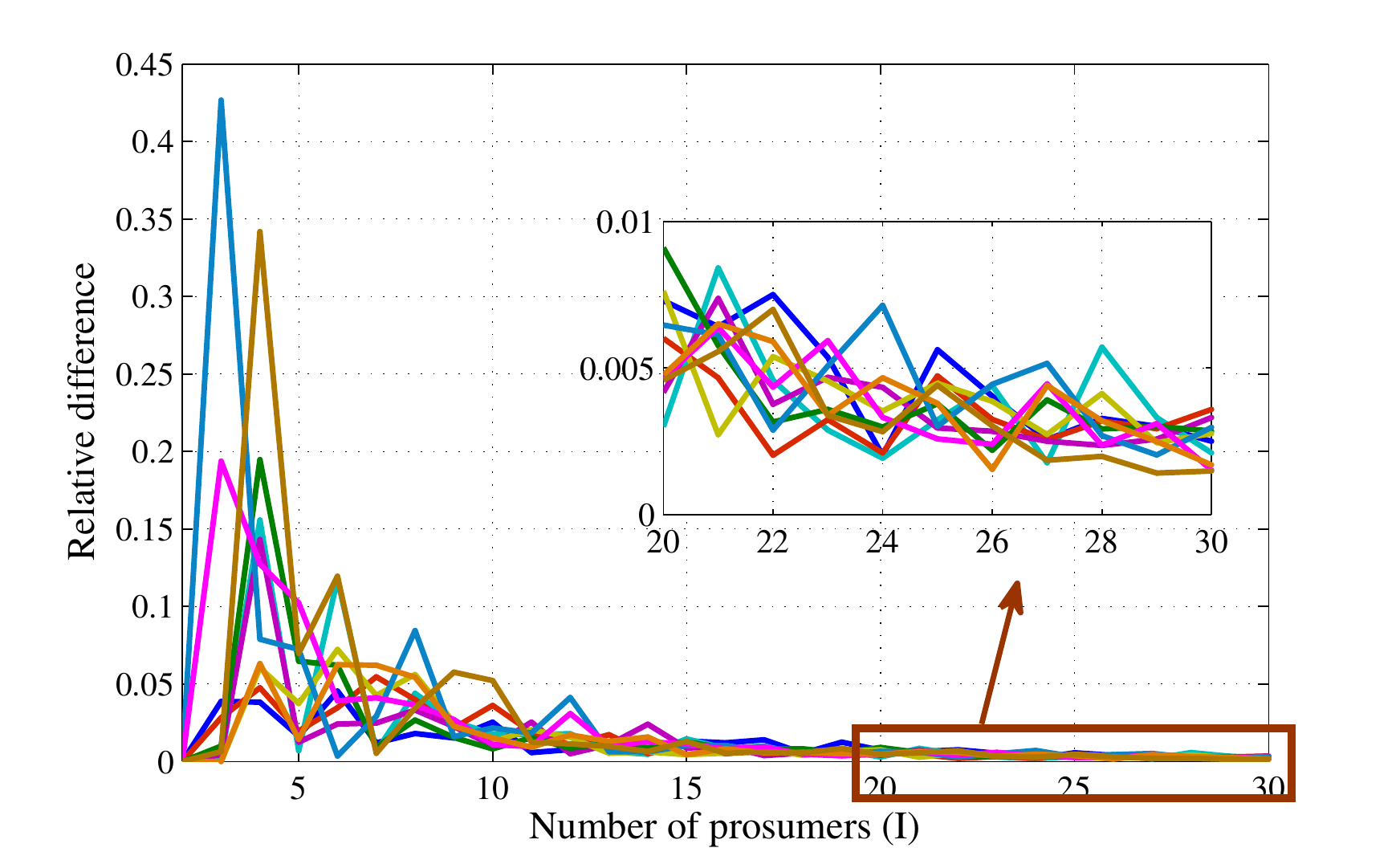}}
\setlength{\abovecaptionskip}{0pt}
\caption{Average relative difference of social total costs between SCO and SMK.}
\setlength{\belowcaptionskip}{-5em}
\label{fig:ARD}
\vspace{-1em}
\end{figure}

\subsection{Impact of Competition}
A special case, in which all prosumers own the same number of resources and are split at the same time, is analyzed in Section V, and Proposition \ref{Thm:prop-6} asserts that competition helps reducing the social total cost. In this subsection, more general cases are considered. A network with 6 lines and 16 resources is tested. First, we change the flow limit to let each line to be congested one-by-one, and then observe how the social total cost will change when competition is introduced to both ends of different lines. The results are shown in Fig. \ref{fig:CP1}. We can find that with the introduction of competition, the total cost always becomes lower. When different lines are congested, the relative relation of introducing competition to different lines remains the same, showing that the position of a line rather than whether it is congested or not is an important factor when we choose where to add competition. 

In the second case, competition is added to node $1 \to 5 \to 4 \to 2 \to 6 \to 7$ sequentially. To eliminate the impact of amplitude, the change of total costs after scaling are presented in Fig.\ref{fig:CP2}. Results show a declining trend in the total cost when the resources are distributed among an increasing number of prosumers.

\begin{figure}[!t]
\centerline{\includegraphics[width=0.8\columnwidth]{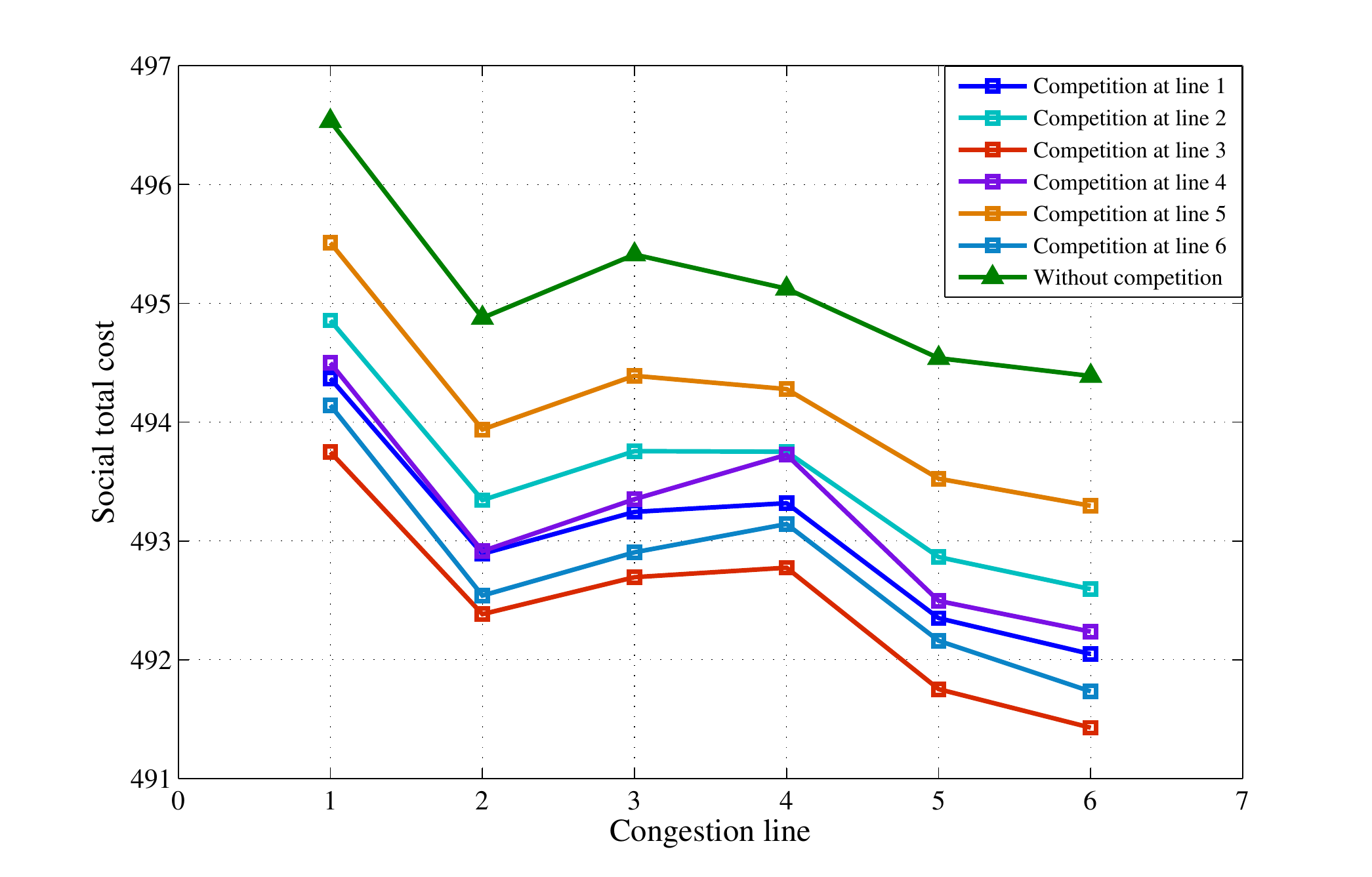}}
\setlength{\abovecaptionskip}{0pt}
\caption{Role of competition when different line is congested.}
\setlength{\belowcaptionskip}{-5em}
\label{fig:CP1}
\vspace{-1em}
\end{figure}

\begin{figure}[!t]
\centerline{\includegraphics[width=0.8\columnwidth]{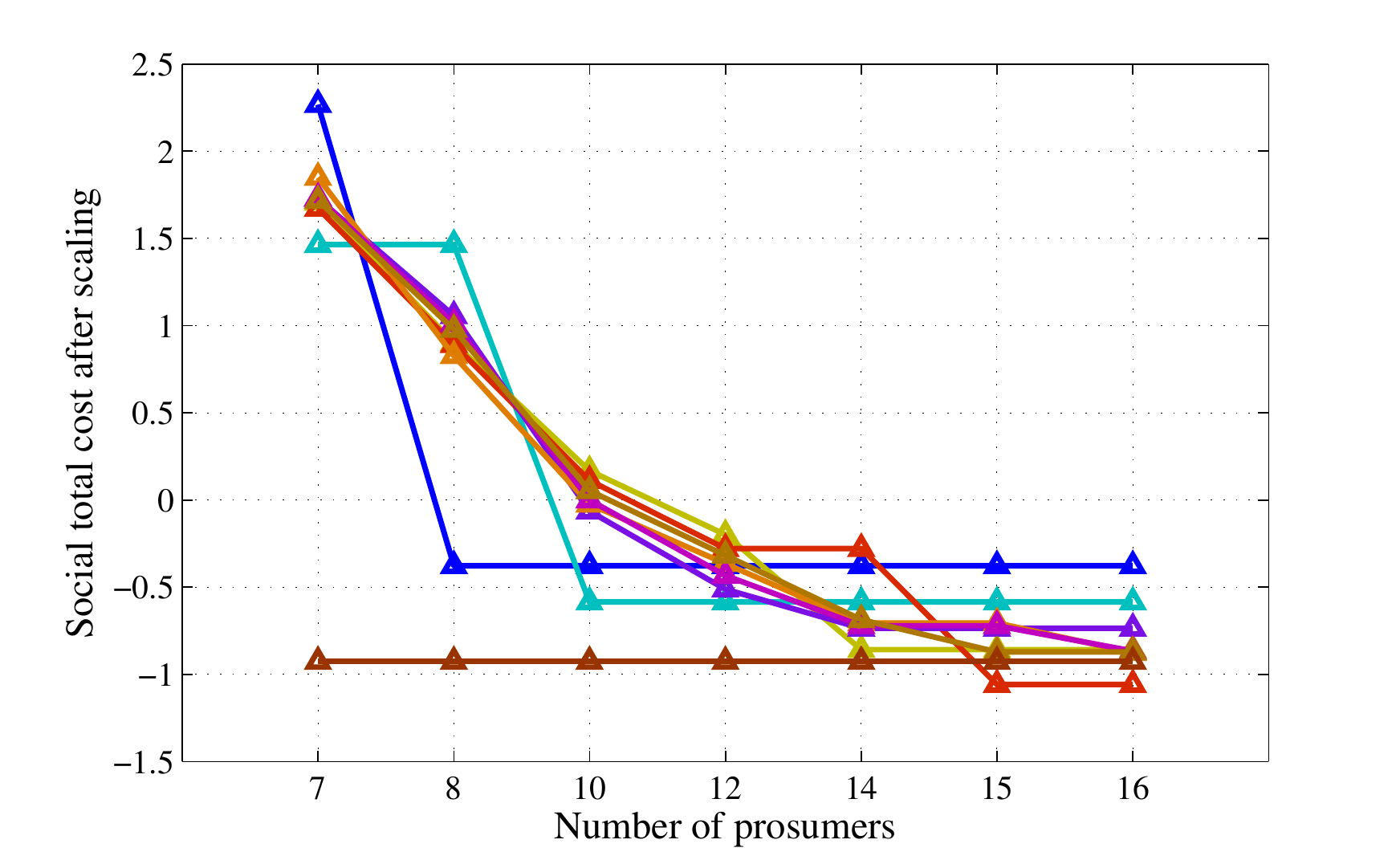}}
\setlength{\abovecaptionskip}{0pt}
\caption{Tendency of total cost when introducing competition.}
\setlength{\belowcaptionskip}{-5em}
\label{fig:CP2}
\vspace{-1em}
\end{figure}

\section{Conclusion}
Prosumers endowed with distributed generation facilities are becoming popular recently, inspiring a new paradigm for energy management via a sharing market. A well-designed sharing mechanism is imperative. This paper comes up with a generic supply-demand function based sharing mechanism considering network constraints and fairness of prices. Price regulation is introduced to restrict market power, ensuring the existence of market equilibrium. Our research discloses some  fundamental properties of the sharing mechanism:

1) The improved sharing mechanism can ensure the existence and uniqueness of a GNE; and a Pareto improvement is achieved among prosumers via sharing.

2) The average cost gap between SMK and SCO tends to zero when the number of prosumers approaches infinity.

3) Competition helps reducing the total social cost and one important factor regarding competition is the location of lines.
  
\ifCLASSOPTIONcaptionsoff
\newpage
\fi

\bibliographystyle{IEEEtran}
\bibliography{IEEEabrv,mybib}

\begin{thebibliography}{10}
\providecommand{\url}[1]{#1}
\csname url@samestyle\endcsname
\providecommand{\newblock}{\relax}
\providecommand{\bibinfo}[2]{#2}
\providecommand{\BIBentrySTDinterwordspacing}{\spaceskip=0pt\relax}
\providecommand{\BIBentryALTinterwordstretchfactor}{4}
\providecommand{\BIBentryALTinterwordspacing}{\spaceskip=\fontdimen2\font plus
\BIBentryALTinterwordstretchfactor\fontdimen3\font minus
  \fontdimen4\font\relax}
\providecommand{\BIBforeignlanguage}[2]{{%
\expandafter\ifx\csname l@#1\endcsname\relax
\typeout{** WARNING: IEEEtran.bst: No hyphenation pattern has been}%
\typeout{** loaded for the language `#1'. Using the pattern for}%
\typeout{** the default language instead.}%
\else
\language=\csname l@#1\endcsname
\fi
#2}}
\providecommand{\BIBdecl}{\relax}
\BIBdecl

\bibitem{da2014impact}
P.~G. Da~Silva, D.~Ili{\'c}, and S.~Karnouskos, ``The impact of smart grid
  prosumer grouping on forecasting accuracy and its benefits for local
  electricity market trading,'' \emph{IEEE Transactions on Smart Grid}, vol.~5,
  no.~1, pp. 402--410, 2014.

\bibitem{rathnayaka2011innovative}
A.~D. Rathnayaka, V.~M. Potdar, and S.~J. Kuruppu, ``An innovative approach to
  manage prosumers in smart grid,'' in \emph{2011 World Congress on Sustainable
  Technologies (WCST)}.\hskip 1em plus 0.5em minus 0.4em\relax IEEE, 2011, pp.
  141--146.

\bibitem{cannon2014uber}
S.~Cannon and L.~H. Summers, ``How uber and the sharing economy can win over
  regulators,'' \emph{Harvard business review}, vol.~13, no.~10, pp. 24--28,
  2014.

\bibitem{zervas2017rise}
G.~Zervas, D.~Proserpio, and J.~W. Byers, ``The rise of the sharing economy:
  Estimating the impact of airbnb on the hotel industry,'' \emph{Journal of
  marketing research}, vol.~54, no.~5, pp. 687--705, 2017.

\bibitem{munoz2017mapping}
P.~Munoz and B.~Cohen, ``Mapping out the sharing economy: A configurational
  approach to sharing business modeling,'' \emph{Technological Forecasting and
  Social Change}, vol. 125, pp. 21--37, 2017.

\bibitem{cheng2016sharing}
M.~Cheng, ``Sharing economy: A review and agenda for future research,''
  \emph{International Journal of Hospitality Management}, vol.~57, pp. 60--70,
  2016.

\bibitem{lam2018more}
C.~Lam and M.~Liu, ``More than taxis with an app: How ride-hailing platforms
  promote market efficiency,'' 2018.

\bibitem{li2016demand}
Z.~Li, Y.~Hong, and Z.~Zhang, ``Do on-demand ride-sharing services affect
  traffic congestion? evidence from uber entry,'' \emph{Evidence from Uber
  Entry (August 30, 2016)}, 2016.

\bibitem{chen2015peeking}
L.~Chen, A.~Mislove, and C.~Wilson, ``Peeking beneath the hood of uber,'' in
  \emph{Proceedings of the 2015 Internet Measurement Conference}.\hskip 1em
  plus 0.5em minus 0.4em\relax ACM, 2015, pp. 495--508.

\bibitem{banerjee2016multi}
S.~Banerjee, D.~Freund, and T.~Lykouris, ``Multi-objective pricing for shared
  vehicle systems,'' \emph{arXiv preprint arXiv:1608.06819}, 2016.

\bibitem{stiglic2015benefits}
M.~Stiglic, N.~Agatz, M.~Savelsbergh, and M.~Gradisar, ``The benefits of
  meeting points in ride-sharing systems,'' \emph{Transportation Research Part
  B: Methodological}, vol.~82, pp. 36--53, 2015.

\bibitem{Sharing-twosided1}
Z.~Fang, L.~Huang, and A.~Wierman, ``Prices and subsidies in the sharing
  economy,'' \emph{Proceedings of the 26th International Conference on World
  Wide Web}, pp. 53--62, 2017.

\bibitem{Sharing-setprice1}
N.~Liu, X.~Yu, C.~Wang, and J.~Wang, ``Energy sharing management for microgrids
  with {PV} prosumers: A stackelberg game approach,'' \emph{IEEE Trans. Ind.
  Inform}, vol.~13, pp. 1088--1098, 2017.

\bibitem{cui2017distributed}
S.~Cui, Y.-W. Wang, and N.~Liu, ``Distributed game-based pricing strategy for
  energy sharing in microgrid with pv prosumers,'' \emph{IET Renewable Power
  Generation}, vol.~12, no.~3, pp. 380--388, 2017.

\bibitem{Sharing-setprice2}
N.~Liu, M.~Cheng, X.~Yu, J.~Zhong, and J.~Lei, ``Energy sharing provider for
  {PV} prosumer clusters: A hybrid approach using stochastic programming and
  stackelberg game,'' \emph{IEEE Trans. Ind. Elec.}, vol.~65, no.~8, pp.
  6740--6750, 2018.

\bibitem{tushar2016energy}
W.~Tushar, B.~Chai, C.~Yuen, S.~Huang, D.~B. Smith, H.~V. Poor, and Z.~Yang,
  ``Energy storage sharing in smart grid: A modified auction-based approach,''
  \emph{IEEE Transactions on Smart Grid}, vol.~7, no.~3, pp. 1462--1475, 2016.

\bibitem{fang2018loyalty}
Z.~Fang, L.~Huang, and A.~Wierman, ``Loyalty programs in the sharing economy:
  Optimality and competition,'' \emph{arXiv preprint arXiv:1805.03581}, 2018.

\bibitem{chen2018energy}
Y.~Chen, W.~Wei, F.~Liu, E.~E. Sauma, and S.~Mei, ``Energy trading and market
  equilibrium in integrated heat-power distribution systems,'' \emph{IEEE
  Transactions on Smart Grid}, 2018.

\bibitem{chen2018optimal}
Y.~Chen, W.~Wei, F.~Liu, M.~Shafie-khah, S.~Mei, and J.~P. Catal{\~a}o,
  ``Optimal contracts of energy mix in a retail market under asymmetric
  information,'' \emph{Energy}, vol. 165, pp. 634--650, 2018.

\bibitem{Sharing-twosided3}
P.~Dutta and A.~Boulanger, ``Game theoretic approach to offering participation
  incentives for electric vehicle-to-vehicle charge sharing,''
  \emph{Transportation Electrification Conference and Expo (ITEC), 2014 IEEE},
  pp. 1--5, 2014.

\bibitem{tang2017double}
L.~Tang, S.~He, and Q.~Li, ``Double-sided bidding mechanism for resource
  sharing in mobile cloud,'' \emph{IEEE Transactions on Vehicular Technology},
  vol.~66, no.~2, pp. 1798--1809, 2017.

\bibitem{cadre2018peer}
H.~L. Cadre, P.~Jacquot, C.~Wan, and C.~Alasseur, ``Peer-to-peer electricity
  market analysis: From variational to generalized nash equilibrium,''
  \emph{arXiv preprint arXiv:1812.02301}, 2018.

\bibitem{chen2018analyzing}
Y.~Chen, W.~Wei, F.~Liu, Q.~Wu, and S.~Mei, ``Analyzing and validating the
  economic efficiency of managing a cluster of energy hubs in multi-carrier
  energy systems,'' \emph{Applied energy}, vol. 230, pp. 403--416, 2018.

\bibitem{makowski1987vickrey}
L.~Makowski and J.~M. Ostroy, ``Vickrey-clarke-groves mechanisms and perfect
  competition,'' \emph{Journal of Economic Theory}, vol.~42, no.~2, pp.
  244--261, 1987.

\bibitem{Sharing-reallocation2}
P.~Chakraborty, E.~Baeyens, K.~Poolla, P.~P. Khargonekar, and P.~Varaiya,
  ``Sharing storage in a smart grid: A coalitional game approach,'' \emph{IEEE
  Trans. Smart Grid, early access}, 2018.

\bibitem{mei2019coalitional}
J.~Mei, C.~Chen, J.~Wang, and J.~L. Kirtley, ``Coalitional game theory based
  local power exchange algorithm for networked microgrids,'' \emph{Applied
  Energy}, vol. 239, pp. 133--141, 2019.

\bibitem{qi2017sharing}
W.~Qi, B.~Shen, H.~Zhang, and Z.-J.~M. Shen, ``Sharing demand-side energy
  resources-a conceptual design,'' \emph{Energy}, vol. 135, pp. 455--465, 2017.

\bibitem{yue2019arxiv}
Y.~Chen, S.~Mei, F.~Zhou, S.~H. Low, W.~Wei, and F.~Liu, ``An energy sharing
  game in prosumers based on generalized demand bidding: Model and
  properties,'' \emph{arXiv:1904.07829 [math.OC]}, 2019.

\bibitem{bose2019some}
S.~Bose and S.~H. Low, ``Some emerging challenges in electricity markets,'' in
  \emph{Smart Grid Control}.\hskip 1em plus 0.5em minus 0.4em\relax Springer,
  2019, pp. 29--45.

\end{thebibliography}

\clearpage
  
\appendices

\makeatletter
\@addtoreset{equation}{section}
\@addtoreset{theorem}{section}
\makeatother
\renewcommand{\theequation}{A.\arabic{equation}}
\renewcommand{\thetheorem}{A.\arabic{theorem}}

\section{Proof of Proposition \ref{Thm:prop-1}}
\begin{proof}
The KKT condition of the platform's sharing clearing problem \eqref{eq:platform} is 
\bsq
\label{eq:lowerKKT}
\bq
 \frac{2}{I}\lambda_i+a\eta+a\sum \limits_l \pi_{il}\alpha_l^{-}-a\sum \limits_l \pi_{il}\alpha_l^{+}=0:\epsilon_i \label{eq:lowerKKT.1}\\ 
  \sum \limits_i (a\lambda_i-b_i)=0:\delta \label{eq:lowerKKT.2}\\
 -F_l \le \sum \limits_i (\pi_{il} (-a\lambda_i+b_i)) \le F_l:\gamma_l^{-},\gamma_l^{+} \label{eq:lowerKKT.3}\\
 \alpha_l^{-} \ge 0, \alpha_l^{+} \ge 0:\phi_l^{-},\phi_l^{+} \label{eq:lowerKKT.4}\\
 \alpha_l^{-}[F_l+\sum \limits_i \pi_{il}(-a\lambda_i+b_i)]=0 :\rho_l^{-} \label{eq:lowerKKT.5}\\
 \alpha_l^{+}[F_l-\sum \limits_i \pi_{il}(-a\lambda_i+b_i)]=0 :\rho_l^{+} \label{eq:lowerKKT.6}
\eq
\esq
After replacing the lower-level problem with its KKT condition, the prosumer $i$'s problem comes down to a mathematical program with equilibrium constraints (MPEC). The sharing game \eqref{eq:platform}-\eqref{eq:sharing-upper} renders an equilibrium problem with equilibrium constraints (EPEC), and $(b^*,\lambda^*)$ is a strong stationary point of the EPEC if for each $i \in \mathcal{I}$, $(b^*_i,\lambda^*)$ is a strong stationary point for the MPEC for prosumer $i$ with fixed $b^*_{-i}$, and satisfies the following KKT conditions.
\bsq
\label{eq:sharing-KKT}
\bq
2c_i^kp_i^k+\beta_i &= & 0 \label{eq:sharing-KKT.1}\\ -2a\lambda_i+b_i-a\beta_i+\frac{2}{I}\epsilon_i+a\delta&& \nonumber\\
+ a\sum \limits_l \pi_{il}\gamma_{l}^{-} -a\sum \limits_l \pi_{il}\gamma_{l}^{+} 
 &= & 0 \label{eq:sharing-KKT.2}\\
\lambda_i+\beta_i-\delta-\sum \limits_l \pi_{il}\gamma_{l}^{-}+\sum \limits_l \pi_{il}\gamma_{l}^{+} & = & 0 \label{eq:sharing-KKT.3}\\
\frac{2}{I}\epsilon_j+a\delta+a\sum \limits_l \pi_{jl}\gamma_l^{-}-a\sum \limits_l \pi_{jl}\gamma_l^{+} & = & 0 \label{eq:sharing-KKT.4}\\
a\sum \limits_i \epsilon_i & = & 0 \label{eq:sharing-KKT.5}\\
a(\sum \limits_l \pi_{il})(\sum \limits_i \epsilon_i)-\phi_{l}^{-}& = & 0 \label{eq:sharing-KKT.6}\\
-a(\sum \limits_l \pi_{il})(\sum \limits_i \epsilon_i)-\phi_{l}^{+} & = & 0 \label{eq:sharing-KKT.7}\\
l \in \mathbb{I}_1: [F_l+\sum \limits_l \pi_{il}(-a\lambda_i+b_i)] & = & 0 \\
l \notin \mathbb{I}_1:  0 \le [F_l+\sum \limits_l \pi_{il}(-a\lambda_i+b_i)]\perp \gamma_l^{-} & \ge & 0 \\
l \in \mathbb{I}_2: [F_l-\sum \limits_l \pi_{il}(-a\lambda_i+b_i)] & = & 0\\
l \notin \mathbb{I}_2: 0 \le [F_l-\sum \limits_l \pi_{il}(-a\lambda_i+b_i)]\perp \gamma_l^{+} & \ge & 0 \label{eq:sharing-KKT.9}\\
l \in \mathbb{I}_3:  \alpha^{-}_l  & = & 0\\
l \notin \mathbb{I}_3: 0 \le \alpha^{-}_l \perp \phi_{l}^{-} & \ge & 0 \label{eq:sharing-KKT.10}\\
l \in \mathbb{I}_4: \alpha^{+}_l & = & 0\\
l \notin \mathbb{I}_4: 0 \le \alpha^{+}_l \perp \phi_{l}^{+} & \ge & 0 \label{eq:sharing-KKT.11} \\
l \in \mathbb{I}_1 \bigcap \mathbb{I}_3: \gamma_l^{-} \ge 0, \phi_{l}^{-} & \ge & 0\\
l \in \mathbb{I}_2 \bigcap \mathbb{I}_4: \gamma_l^{+} \ge 0, \phi_{l}^{+} & \ge & 0
\eq
\esq
Together with the balancing constraint in \eqref{eq:sharing-upper} and \eqref{eq:lowerKKT}. Here, the sets $\mathbb{I}_1$, $\mathbb{I}_2$, $\mathbb{I}_3$ and $\mathbb{I}_4$ are defined as
\bq
& \mathbb{I}_1:=\{l|F_l+\sum \limits_l \pi_{il}(-a\lambda_i+b_i)=0\} \nonumber\\
& \mathbb{I}_2:=\{l|F_l-\sum \limits_l \pi_{il}(-a\lambda_i+b_i)=0\} \nonumber\\
& \mathbb{I}_3:=\{l|\alpha^{-}_l=0\} \nonumber\\
& \mathbb{I}_4:=\{l|\alpha^{+}_l=0\} \nonumber
\eq

A stationary point of the EPEC problem \eqref{eq:platform}-\eqref{eq:sharing-upper} corresponds to an isolated GNE of the sharing game. However, even though the KKT condition \eqref{eq:sharing-KKT} is not met, it is still possible that $b^*$ is a GNE, but not an isolated one.  

The KKT condition of problem \eqref{eq:central} is
\bsq
\label{eq:central-KKT}
\bq
2c_i^kp_i^k+\frac{\sum_k p_i^k-D_i}{a(I-1)}+\kappa+\sum \limits_l \pi_{il}\tau_l^{-}-\sum \limits_l \pi_{il}\tau_l^{+}=0\\
\sum \limits_i \sum \limits_k p_i^k =\sum \limits_i D_i\\
0 \le (\sum \limits_i \pi_{il}(D_i-\sum \limits_k p_i^k)+F_l) \perp \tau_l^{-} \ge 0\\
0 \le (-\sum \limits_i \pi_{il}(D_i-\sum \limits_k p_i^k)+F_l) \perp \tau_l^{+} \ge 0
\eq
\esq

If $b^*$ is an isolated GNE of the sharing game and $p^*$ the corresponding optimal adjustment. Then it satisfies the KKT conditions \eqref{eq:sharing-KKT}. If we add up \eqref{eq:sharing-KKT.2} and  \eqref{eq:sharing-KKT.4} for all $j \ne i$, together with \eqref{eq:sharing-KKT.5}, we have
\bq
\label{eq8}
-2a\lambda_i^*+b_i^*-a\beta_i^*+aI\delta^* \nonumber &&\\
+aI\sum \limits_l \pi_{il}\gamma_{l}^{*-}-aI\sum \limits_l \pi_{il}\gamma_{l}^{*+} &=& 0
\eq  
\eqref{eq8}+$aI$\eqref{eq:sharing-KKT.3} gives
\bq
a(I-1)\lambda_i^*+D_i-\sum \limits_k p_i^{k*}+a(I-1)\beta_i^*=0
\eq
With \eqref{eq:sharing-KKT.1} we have
\bq
2c_i^kp_i^{k*}+\frac{\sum_k p_i^{k*}-D_i}{a(I-1)}=\lambda_i^*
\eq
From \eqref{eq:lowerKKT.1}, we know that 
\bq
\frac{2}{I}[2c_i^kp_i^{k*}+\frac{\sum_k p_i^{k*}-D_i}{a(I-1)}]+a\sum \limits_l \pi_{il}\alpha_l^{*-}-a\sum \limits_l \pi_{il}\alpha_l^{*+}=-a\eta^*
\eq
which is a constant among all prosumers. Let $\kappa=\frac{aI}{2}\eta^*$, $\tau_l^{\pm}=\frac{aI}{2}\alpha_l^{*\pm}$, it is easy to check that the KKT conditions \eqref{eq:central-KKT} are met. And obviously we have 
$b_i^*=\tilde b_i(p_i^*)$. Since problem \eqref{eq:central} is a strictly convex optimization problem, its optimal solution is unique, so as the isolated GNE. This completes the proof.
\end{proof}

\renewcommand{\theequation}{B.\arabic{equation}}
\renewcommand{\thetheorem}{B.\arabic{theorem}}
\section{Proof of Proposition \ref{Thm:prop-2}}
\begin{proof}
Define $y_i=-a\lambda_i+b_i$, then the sharing problem is equivalent to Prosumer $i$'s decision-making:
\bsq
\label{eq:sharing-upper-m}
\bq
\mathop {\min} \limits_{p_i^k,\forall k,b_i} && \sum \limits_k c_i^k (p_i^k)^2+\frac{1}{a}y_i(b)(-y_i(b)+b_i) \label{eq:sharing-upper-m.1}\\
\rm{s.t.}&& \sum \limits_k p_i^k+y_i=D_i \label{eq:sharing-upper-m.2}
\eq
\esq

Sharing platform's problem:
\bsq
\label{eq:sharing-lower-m}
\bq
\mathop {\min} \limits_{y_i,\forall i} && \sum \limits_i (y_i-b_i)^2 \label{eq:sharing-lower-m.1}\\
\rm{s.t.}&& \sum \limits_i y_i=0 \label{eq:sharing-lower-m.2}\\
&& -F_l \le \sum \limits_i \pi_{il} y_i \le F_l  \label{eq:sharing-lower-m.3}
\eq
\esq

 Before proving Proposition \ref{Thm:prop-2}, we first give the following lemmas. Fix other prosumers' strategies $b_{-i}$ and varies $b_i$. Suppose $y_i^*(b_i)$ is the optimal solution given by \eqref{eq:sharing-lower-m}. Denote the objective function of \eqref{eq:sharing-lower-m} as $g(y,b)$ and $g(y,b)=||y-b||_2$.
\begin{lemma}
\label{lemma-1}
$y_i^*(b_i)$ is continuous in $b_i$.
\end{lemma}
\begin{proof}
Recall the \emph{Maximum Theorem}, which says: Let $Y$ and $B_i$ be metric spaces, $g:Y \times B_i \rightarrow \mathbb{R}$ be a function jointly continuous in its two arguments, and $C: B_i \twoheadrightarrow Y$  be a compact-valued correspondence, let
$$
y^*(b_i)={\rm{argmax}}\{-g(y,b_i)|y \in C(b_i)\}
$$
If $C$ is continuous at some point $b_i$, then $y^*$ is non-empty, compact-valued, and upper hemi-continuous at $b_i$.

Here, first $g(y,b_i,b_{-i})$ is continuous both in $y$ and $b_i$. The feasible region of \eqref{eq:sharing-lower-m} is independent of $b_i$ and thus $C$ is a a compact-valued correspondence. Therefore, $y^*(b_i)$ is upper hemi-continuous at $b_i$. As \eqref{eq:sharing-lower-m} is a strictly convex optimization problem and has a unique solution, so upper hemi-continuous implies continuous.

\end{proof}
\begin{lemma}
\label{lemma-2}
We have
\bq
0 \le \frac{\Delta y_i^*}{\Delta b_i} \le \frac{I-1}{I}
\eq
\end{lemma}
\begin{proof}
For prosumer $i$, fix $b_{-i}$. Assume that $b_i$ changes to $b_i+\Delta b_i$ and $y_i^*$ changes to $y_i^*+\Delta y_i^*$ accordingly. Without loss of generality, suppose $\Delta b_i>0$. 

First, we prove the left hand side inequality. If $\Delta y_i^*<0$, then construct $y^{'}=y^*+\frac{1}{2}\Delta y^*$. Then 
\bq
&& g(y^*,b)+g(y^*+\Delta y^*, b+\Delta b) \nonumber\\
= && (y_i^*-b_i)^2+\sum \limits_{j \ne i}(y_j^*-b_j)^2+ (y_i^*+\Delta y_i^*-b_i-\Delta b_i)^2 \nonumber\\
&& +\sum \limits_{j \ne i}(y_j^*+\Delta y_j^*-b_j)^2
\eq
\bq
\!\!&& \!\!\!\!\! g(y^*+\frac{1}{2}\Delta y^*,b)+g(y^*+\frac{1}{2}\Delta y^*,b+\Delta b) \nonumber\\
= \!\!&& \!\!\!\!\!(y_i^*+\frac{1}{2}\Delta y_i^*-b_i)^2+\sum \limits_{j \ne i} (y_j^*+\frac{1}{2}\Delta y_j^*-b_j)^2 \nonumber\\
\!\!&& \!\!\!\!\! + (y_i^*+\frac{1}{2}\Delta y_i^*-b_i-\Delta b_i)^2+\sum \limits_{j \ne i} (y_j^*+\frac{1}{2}\Delta y_j^*-b_j)^2
\eq
According to Jensen Inequality, we have
\bq
\frac{(y_j^*-b_j)^2+(y_j^*+\Delta y_j^*-b_j)^2}{2} \ge (y_j^*+\frac{1}{2}\Delta y_j^*-b_j)^2
\eq
and because
\bq
&& (y_i^*-b_i)^2+(y_i^*+\Delta y_i^*-b_i-\Delta b_i)^2-(y_i^*+\frac{1}{2}\Delta y_i^*-b_i)^2 \nonumber\\
&& -(y_i^*+\frac{1}{2}\Delta y_i^*-b_i-\Delta b_i)^2 \nonumber\\
&=& \frac{(\Delta y_i^*)^2}{2}-\Delta y_i^* \Delta b_i >0
\eq
As a result, we have 
\bq
&& g(y^*,b)+g(y^*+\Delta y^*, b+\Delta b) \nonumber\\
&> & g(y^*+\frac{1}{2}\Delta y^*,b)+g(y^*+\frac{1}{2}\Delta y^*,b+\Delta b)
\eq
which is contradict to the assumption that $y^* $ is the optimal solution corresponding to $b$ and $y^*+\Delta y^*$ is the optimal solution corresponding to $b+\Delta b$. So we have $\Delta y_i^*/\Delta b_i \ge 0$.

Next, we prove the right hand side inequality. If we have $\Delta y_i^*/\Delta b_i > (I-1)/I$. First, because the feasible region of \eqref{eq:sharing-lower-m} is a polyhedron and $g(y,b)=||y-b||_2$, so the optimal $y$ is piece-wise linear in $b$. In each segment with the direction of $\Delta y^{'}$, $y^{'}$ is the optimal corresponding to $b^{'}$ when and only when $(y^{'}-b^{'})\perp \Delta y^{'}$. When $\Delta b$ is small enough, $y^*$ and $y^*+\Delta y^*$ fall in the same segment. Due to the optimality, we have $(y^*-b) \perp \Delta y^*$ and $(y^*+\Delta y^*-b^*-\Delta b^*) \perp \Delta y^*$.

Denote a vector $\tilde{y}=(y_1^*-\frac{1}{I}\Delta b_i,...,y_i^*+\frac{I-1}{I}\Delta b_i,,...,y_I^*-\frac{1}{I}\Delta b_i)$. Suppose $\check y=(y_1^*+\alpha \Delta y_1^*,..., y_i^*+\alpha \Delta y_i^*, ..., y_I^*+\alpha \Delta y_I^*)$. $\alpha$ is a parameter such that $(\check y-\tilde{y}) \perp \Delta y^*$, which means
\bq
&& (\alpha \Delta y_1^*+\frac{1}{I}\Delta b_i,...,\alpha \Delta y_i^*-\frac{I-1}{I}\Delta b_i,,...,\alpha \Delta y_I^*+\frac{1}{I}\Delta b_i) \nonumber\\
& \cdot & (\Delta y_1^*,..., \Delta y_i^*,...,\Delta y_I^*) =0
\eq
We can get
\bq
\label{eq:alpha}
0 \le \alpha & = &\frac{\Delta b_i \Delta y_i^*}{\sum \limits_j (\Delta y_j^*)^2} \nonumber\\
& < & \frac{\Delta b_i\Delta y_i^*}{(\Delta y_i^*)^2+\frac{1}{I-1}(\Delta y_i^*)^2} \nonumber\\
& < & 1
\eq
So that $\check{y}$ is a feasible point for problem \eqref{eq:sharing-lower-m} and $\check{y} \ne (y^*+\Delta y^*)$. Since $(y^*-b)\perp\Delta y^*$ and $\sum_i \Delta y_i^*=0$, so we have
\bq
&& [(b+\Delta b)-\tilde{y}] \cdot \Delta y^* \nonumber\\
& = & (b_1-y_1^*+\frac{\Delta b_i}{I},...,b_I-y_I^*+\frac{\Delta b_i}{I}) \nonumber\\
&& \cdot (\Delta y_1^*, \Delta y_2^*,...,\Delta y_I^*) \nonumber\\
& = & 0
\eq
As a result, we have $(\check{y}-b-\Delta b) \perp \Delta y^*$ and $(y^*+\Delta y^*-b^*-\Delta b^*) \perp \Delta y^*$, which is contradict to  $\check{y} \ne (y^*+\Delta y^*)$. This completes the proof.
\end{proof}

With \textbf{Lemma \ref{lemma-1}} and \textbf{Lemma \ref{lemma-2}}, we can prove Proposition \ref{Thm:prop-2}. $u_i(\lambda_i,b_i,p_i)$ can be rewritten as
\bq
&& u_i(y_i,b_i,p_i) \nonumber\\
& = & max\{\frac{1}{a}y_i(-y_i+b_i),(md_i+\frac{\sum_k p_i^k-D_i}{a(I-1)})y_i\} \nonumber
\eq 

We only prove the case when $y_i>0$, the other situation when $y_i \le 0$ can be proven in the same way. Because $y_i^*(b_i)$ is a piece-wise linear function of $b_i$, so in each segment, $y_i^*(b_i)$ is derivable. 

If $y_i>0$ and $\frac{1}{a}y_i(-y_i+b_i)>( md_i+\frac{p_i-D_i}{a(I-1)})(D_i-p_i)$, which means
\bq
\frac{1}{a}(-y_i+b_i)>2\bar c_i (D_i-y_i)-\frac{y_i}{a(I-1)}
\eq
The left-hand side is increasing in $b_i$ while the right-hand side is decreasing in $b_i$, two parts are equal when and only when $b_i=b_i^*$.
Then if $b_i>b_i^*$, we have $\Gamma_i(y_i,b_i,p_i)=(\sum_k c_i^k(p_i^k)^2+\frac{1}{a}y_i(-y_i+b_i))_{max}=\bar c_i(D_i-y_i)^2+\frac{1}{a}y_i(-y_i+b_i)$.
\bq
&& \frac{d \Gamma_i}{d b_i} \nonumber\\
&= & -2\bar c_i(D_i-y_i)\frac{dy_i}{db_i}-\frac{2}{a}y_i\frac{dy_i}{db_i}+\frac{y_i}{a}+\frac{b_i}{a}\frac{dy_i}{db_I} \nonumber\\
& \ge & -\frac{y_i}{a(I-1)}\frac{dy_i}{db_i}-\frac{y_i}{a}\frac{dy_i}{db_i}+\frac{y_i}{a} \nonumber\\
& \ge & -\frac{y_i}{a}+\frac{y_i}{a} \ge 0
\eq
It means when $b_i>b_i^*$, in each segment, the cost function is increasing; and because the cost function is continuous, so $\Gamma_i$ is increasing when $b_i>b_i^*$.

If $b_i<b_i^*$, then $\Gamma_i(y_i,b_i,p_i)=\sum_k c_i^k(p_i^k)^2+(md_i+\frac{\sum_k p_i^k-D_i}{a(I-1)})(D_i-p_i)=\bar c_i(D_i-y_i)^2+(2\bar c_i(D_i-y_i)-\frac{y_i}{a(I-1)})y_i$.
\bq
&& \frac{d\Gamma_i}{db_i} \nonumber\\
& = & -2\bar c_i(D_i-y_i)\frac{dy_i}{db_i}+2\bar c_iD_i\frac{dy_i}{db_i}-4\bar c_iy_i\frac{dy_i}{db_i}-\frac{2y_i}{a(I-1)}\frac{dy_i}{db_i}\nonumber\\
& = & -2\bar c_iy_i\frac{dy_i}{db_i}-\frac{2y_i}{a(I-1)}\frac{dy_i}{db_i} \le 0
\eq
This implies that when $b_i<b_i^*$, $\Gamma_i$ is decreasing. In conclusion, the cost $\Gamma_i$ reaches its minimum at $b_i=b_i^*$.

We have proved that $b^*$ is a GNE of the sharing game. Next, we prove that it is the unique one. Given any $b_{-i}$, it has been proved that the optimal strategy of prosumer $i$ is to choose a $b_i$ such that the resulting $y_i$, $p_i$ satisfies that
$$b_i=2a\bar{c}_i\sum_k p_i^k+\frac{I-2}{(I-1)}(D_i-\sum_k p_i^k), \;\; y_i=D_i-\sum_k p_i^k$$
Then we show that this can always be achieved. Assume $y^{*'}$ is the optimal solution of the following optimization problem
\bq
\mathop{\min} \limits_i && \frac{(\frac{1}{I-1}y_i+2a\bar{c}_iy_i-2a\bar{c}_iD_i)^2}{2a\bar{c}_i+\frac{1}{I-1}}+\sum \limits_{j \ne i} (y_j-b_j)^2 \\
\rm{s.t.} &&  \sum \limits_i y_i=0\\
&& -F_l \le \sum \limits_i \pi_{il}y_i \le F_l
\eq
If we let $b_i=2a\bar{c}_i(D_i-y_i)+\frac{I-2}{(I-1)}y_i$, it is easy to check that $y_i-b_i=(\frac{1}{I-1}+2a\bar{c}_i)y_i-2a\bar{c}_iD_i$.

By comparing the KKT condition of the lower level problem \eqref{eq:sharing-lower-m} and problem \eqref{eq:central2}, it is easy to show that $p^*$ is the optimal solution of \eqref{eq:central2}, which is unique. This completes the proof.
\end{proof}

\renewcommand{\theequation}{C.\arabic{equation}}
\renewcommand{\thetheorem}{C.\arabic{theorem}}
\section{Proof of Proposition \ref{Thm:prop-2-2}}
\begin{proof}
When there is no congestion, it degenerates to the situation in \cite{yue2019arxiv}, where a unique NE always exists. When there is congestion,
if there exists an IGNE, according to Proposition \ref{Thm:prop-1}, we have $\frac{1}{a}y_i^*(-y_i^*+b_i^*)=(2\bar c_i\sum_k p_i^{k*}+\frac{\sum_k p_i^{k*}-D_i}{a(I-1)})(D_i-\sum_k p_i^{k*})$. The derivative $\frac{d\Gamma_i}{db_i}$ must be zero at the optimal point. Thus, we have
\bq
&& \frac{d \Gamma_i}{d b_i} \nonumber\\
&= & -2\bar c_i(D_i-y_i)\frac{dy_i}{db_i}-\frac{2}{a}y_i\frac{dy_i}{db_i}+\frac{y_i}{a}+\frac{b_i}{a}\frac{dy_i}{db_I} \nonumber\\
& = & -\frac{y_i}{a(I-1)}\frac{dy_i}{db_i}-\frac{y_i}{a}\frac{dy_i}{db_i}+\frac{y_i}{a} 
\eq
 $\frac{d\Gamma_i}{db_i}=0$ when and only when $\frac{dy_i^*}{db_i^*}=\frac{I-1}{I}$. According to \eqref{eq:alpha}, we have $\frac{dy_j^*}{db_i^*}=-\frac{1}{I}$. Symmetrically, we have
 \bq
 \frac{dy_i^*}{db_i^*}=\frac{I-1}{I},\forall i \nonumber\\
 \frac{dy_j^*}{db_i^*}=\frac{-1}{I},\forall i, \forall j \ne i
 \eq
 Without loss of generality, assume the binding constraint is
 \bq
 \sum \limits_i \pi_{il^{'}} y_i =F_{l^{'}}
 \eq                                                                                                       
 Take the derivative with respect to $b_i,\forall i$, it is easy to conclude that $\pi_{il^{'}},\forall i$ equal the same $\pi_{l^{'}}$. Suppose line $l^{'}$ links the bus $n$ and $m$, then 
 \bq
 \label{eq:contradict}
 [\pi_{l^{'}},...,\pi_{l^{'}}]=[0,...,\mathop 1\limits_n ,...,\mathop { - 1}\limits_m ,...,0]B^{-1}
 \eq
 where $B$ is a full-rank matrix with $B_{ij}=-\frac{1}{x_{ij}}$ and $B_{ii}=\frac{1}{x_{i0}}+\sum_{j \ne i}\frac{1}{x_{ij}}$, $x_{ij}>0,\forall i,j$. With \eqref{eq:contradict} we can get that
 \bq
  [\pi_{l^{'}},...,\pi_{l^{'}}]B=[0,...,\mathop 1\limits_n ,...,\mathop { - 1}\limits_m ,...,0]
 \eq
 which implies $x_{i0}=\frac{1}{\pi_{l^{'}}}$ and $x_{j0}=-\frac{1}{\pi_{l^{'}}}$, which is contradict to $x_{ij}>0,\forall i,j$. As a result, when there is congestion, no IGNE exists. This completes the proof.
\end{proof}

\renewcommand{\theequation}{D.\arabic{equation}}
\renewcommand{\thetheorem}{D.\arabic{theorem}}
\section{Proof of Proposition \ref{Thm:prop-3}}
\begin{proof}
Given other prosumers' strategies ($j \ne i$), first we solve the following optimization problem
\bsq
\label{eq:sharing-N-1}
\bq
\mathop {\min} \limits_{\lambda_j,j \ne i} && \frac{1}{I}\sum \limits_{j \ne i} \lambda_j^2\\
&& \sum \limits_{j \ne i} (a\lambda_j-b_j)=0:\eta^{'}\\
&& -F_l \le \sum \limits_{j \ne i} \pi_{jl} (-a\lambda_j+b_j) \le F_l:\alpha_l^{-'},\alpha_l^{+'}
\eq
\esq
Suppose the optimal solution and the corresponding dual variables are $\lambda_j^*$, $\eta^{'*}$,$\alpha_l^{\pm'*}$. Then if prosumer $i$ bids $b_i=\frac{a^2I}{2}(-\eta^{'*}-\sum_l \pi_{il}\alpha_l^{-'*}+\sum_l \pi_{il}\alpha_l^{+'*})$, then by comparing the KKT condition \eqref{eq:lowerKKT.1}, it is easy to check the corresponding $a\lambda_i=-b_i$. At this time, the cost of prosumer $i$ is the same as when he makes decision individually. Since each prosumer $i$ solves an minimization problem, a Pareto improvement can be achieved. This completes the proof.
\end{proof}

\renewcommand{\theequation}{E.\arabic{equation}}
\renewcommand{\thetheorem}{E.\arabic{theorem}}
\section{Proof of Proposition \ref{Thm:prop-4}}
\begin{proof}
If $b^*$ is the GNE of the sharing game \eqref{eq:platform} and \eqref{eq:modified-obj}, then it satisfies the KKT conditions \eqref{eq:sharing-KKT}. Then in Appendix A, we prove show that 
\bq
\lambda_i^* & = & 2c_i^kp_i^k+\frac{\sum_k p_i^k-D_i}{a(I-1)} \nonumber\\
& = & \frac{aI}{2}[-\eta^*-\sum \limits_l \pi_{il}\alpha_l^{*-}+\sum \limits_l \pi_{il}\alpha_l^{*+}] \nonumber\\
& = & -\kappa^*-\sum \limits_l \pi_{il}\tau_l^{*-}+\sum \limits_l \pi_{il}\tau_l^{*+} 
\eq
This completes the proof.
\end{proof}

\renewcommand{\theequation}{F.\arabic{equation}}
\renewcommand{\thetheorem}{F.\arabic{theorem}}
\section{Proof of Proposition \ref{Thm:prop-5}}
\begin{proof}
As $b^*$ is the GNE of the sharing game \eqref{eq:platform} and \eqref{eq:modified-obj}, $p^*$ the corresponding optimal adjustment, according to Proposition \ref{Thm:prop-1}, $p^*$ is the optimal solution of problem \eqref{eq:central2}. $\bar p$ is the optimal solution of problem \eqref{eq:AGG}. The feasible region of problem \eqref{eq:central2} and \eqref{eq:AGG} are the same, so $p^*$ is a feasible solution of problem \eqref{eq:AGG}. Due to the optimality, we have
$$\sum \nolimits_{i \in \mathcal{I}} f_i(p_i^*(I)) \ge \sum \nolimits_{i \in \mathcal{I}} f_i(\bar p_i(I))$$

Then, we prove the second statement. The net injection of prosumer $i$ is $\sum_k p_i^k-D_i$, which is sum of the energy flows of lines connected to $i$. Denote $\hat{F}:=\max\{F_l,\forall l\}$. According to Assumption \ref{assumption-1}, the degree of node $i$ is no more than $G^D$ and the flow of each line is within $[-\hat F,\hat F]$, so
\bq
-G^D\hat{F} \le \sum_k p_i^k-D_i \le G^D\hat{F}
\eq

For any given $\epsilon>0$, we choose a large enough number $I_0:=\frac{a\epsilon}{G^D\hat{F}}+1$. Then for arbitrary number $I>I_0$, there is
\bq
&& \frac{1}{I}[\sum \limits_i f_i(p_i^*(I))-\sum \limits_i f_i(\bar p_i(I))] \nonumber\\
& = &  \frac{1}{I}\sum \limits_i [f_i(p_i^*(I))+\frac{(D_i-\sum_k p_i^{k*})^2}{2a(I-1)} \nonumber\\
&& -f_i(\bar p_i(I))- \frac{(D_i-\sum_k p_i^{k*})^2}{2a(I-1)}] \nonumber\\
& \le & \frac{1}{I}\sum \limits_i [f_i(\bar p_i(I))+\frac{(D_i-\sum_k \bar p_i^k)^2}{2a(I-1)} \nonumber\\
&& - f_i(\bar p_i(I))- \frac{(D_i-\sum_k p_i^{k*})^2}{2a(I-1)}] \nonumber\\
& = & \frac{1}{I}\frac{1}{2a(I-1)}\sum \limits_i  [(D_i-\sum_k \bar p_i^k)^2-(D_i-\sum_k p_i^{k*})^2] \nonumber\\
& \le & \frac{1}{I}\frac{1}{2a(I-1)} \sum \limits_i [|D_i-\sum_k \bar p_i^k|^2+|D_i-\sum_k p_i^{k*}|^2] \nonumber\\
& \le & \frac{G^D\hat{F}}{a(I-1)} \le \epsilon
\eq
This completes the proof.
\end{proof}

\renewcommand{\theequation}{G.\arabic{equation}}
\renewcommand{\thetheorem}{G.\arabic{theorem}}
\section{Proof of Proposition \ref{Thm:prop-6}}
\begin{proof}
The KKT condition of the sharing game under competition is as follows.
\bsq
\label{eq:competition-KKT}
\bq
2c_i^kp_i^k+\frac{\sum \limits_k p_i^k-D_i}{a(I-1)}+\kappa+\sum \limits_l \pi_{il}\tau_l^{-}-\sum \limits_l \pi_{il}\tau_l^{+}=0\\
\sum \limits_i \sum \limits_k p_i^k =\sum \limits_i D_i\\
0 \le (\sum \limits_i \pi_{il}(D_i-\sum \limits_k p_i^k)+F_l) \perp \tau_l^{-} \ge 0\\
0 \le (-\sum \limits_i \pi_{il}(D_i-\sum \limits_k p_i^k)+F_l) \perp \tau_l^{+} \ge 0
\eq
\esq
If the resource prosumer $i$ owns is equally distributed to $M$ prosumers ($K$ is divisible by $M$), and satisfies that
\bq
 \sum \limits_{j=\frac{(m-1)K}{M}+1}^{\frac{mK}{M}} p_{i}^j- D_{i}^m & = & \frac{1}{M} (\sum \limits_{k=1}^K p_i^k- D_i), m=\{1,2,...,M\},\forall i \nonumber
\eq
Then it is easy to find that 
\bsq
\bq
2c_i^j p_i^j +\frac{\sum \limits_{j=\frac{(m-1)K}{M}+1}^{\frac{mK}{M}} p_{i}^j-D_i^m}{a(I-1)/M} +\kappa+\sum \limits_l \pi_{il}\tau_l^{-} \nonumber\\
-\sum \limits_l \pi_{il}\tau_l^{+}=0,m=\{1,2,...,M\},\forall i \\
\sum \limits_i \sum \limits_k p_i^k =\sum \limits_i \sum \limits_m D_i^m \\
0 \le (\sum \limits_i \pi_{il}(\sum \limits_m D_i^m- \sum \limits_k p_i^k)+F_l) \perp \tau_l^{-} \ge 0 \\
0 \le (-\sum \limits_i \pi_{il}(\sum \limits_m D_i^m- \sum \limits_k p_i^k)+F_l) \perp \tau_l^{+} \ge 0
\eq
\esq
which means $p_i^j$ is the solution of the following optimization problem
\bsq
\label{eq:central-2}
\bq
\mathop {\min} \limits_{p_i,\forall i} && \sum \limits_i \sum \limits_m \sum \limits_j c_i^j(p_i^j)^2+\sum \limits_i \sum \limits_m \frac{(D_i^m-\sum \limits_{j=\frac{(m-1)K}{M}+1}^{\frac{mK}{M}}p_i ^j)^2}{2a(I-1)/M} \nonumber\\
\rm{s.t.} && \sum \limits_i \sum \limits_k p_i^k =\sum \limits_i\sum \limits_m D_i^m \nonumber\\
&& -F_l \le \sum \limits_i \pi_{il}(\sum \limits_m D_i^m-\sum \limits_k p_i^k) \le F_l 
\eq
\esq

Denote
$$f(p)=\sum \limits_i \sum \limits_m \sum \limits_j c_i^j(p_i^j)^2$$
$$w(p)=\sum \limits_i \sum \limits_m \frac{(D_i^m-\sum \limits_{j=\frac{(m-1)K}{M}+1}^{\frac{mK}{M}}p_i ^j)^2}{2a(I-1)/M}$$
$$h(p)=\sum \limits_i \sum \limits_m \frac{(D_i^m-\sum \limits_{j=\frac{(m-1)K}{M}+1}^{\frac{mK}{M}}p_i ^j)^2}{2a(MI-1)}$$
Then $w(p)=\frac{(I-1)}{(MI-1)M}h(p)$. Suppose $p^{1*}$ the optimal solution before partition, and $p^{2*}$ the optimal solution after partition. If we have $f(p^{1*}) \le f(p^{2*})$. Due to optimality, we have
\bq
f(p^{1*})+w(p^{1*}) \le f(p^{2*})+w(p^{2*}) \nonumber
\eq
which means
\bq
w(p^{1*})-w(p^{2*}) & \le  & f(p^{2*})-f(p^{1*}) \nonumber\\
& \le & \frac{(MI-1)M}{I-1} [f(p^{2*})-f(p^{1*})] \nonumber
\eq
which means
\bq
h(p^{1*})-h(p^{2*}) & = &  \frac{I-1}{(MI-1)M} [w(p^{1*})-w(p^{2*})] \nonumber\\
& \le &  f(p^{2*})-f(p^{1*}) \nonumber
\eq
showing  that
\bq
f(p^{1*})+h(p^{1*}) \le f(p^{2*})+h(p^{2*}) \nonumber
\eq
which is contradict to the assumption that $p^{2*}$ is the optimal solution after partition. As a result, after partition, the total cost decreases.
\end{proof}

\end{document}